\documentclass[11pt,reqno]{amsart}
\usepackage{amsmath,amscd,amssymb,amsfonts,amsthm,courier,relsize,bm,tikz}
\usetikzlibrary{patterns}
\usepackage{hyperref,enumitem,mathrsfs,slashed,mathtools,graphicx}

\usepackage{xcolor}  
\hypersetup{
    colorlinks,
    linkcolor={red!50!black},
    citecolor={blue!70!black},
    urlcolor={blue!80!black}
}

\usepackage[cmtip,matrix,arrow]{xy}

\newtheorem{theorem}{Theorem}[section]
\newtheorem{corollary}[theorem]{Corollary}
\newtheorem{proposition}[theorem]{Proposition}

\newtheorem{lemma}[theorem]{Lemma}
\theoremstyle{definition}    
\newtheorem{definition}[theorem]{Definition}
\theoremstyle{remark}

\newtheorem{remark}[theorem]{Remark}

\newcommand{\pair}[2]{\langle #1, #2 \rangle}
\newcommand{\ignore}[1]{}

\newcommand{\ol}[1]{\overline{#1}}
\newcommand{\ul}[1]{\underline{#1}}
\newcommand{\ti}[1]{\widetilde{#1}}
\newcommand{\wh}[1]{\widehat{#1}}
\newcommand{\st}[1]{\mathsf{#1}}
\newcommand{\scr}[1]{\mathscr{#1}}

\newcommand{\tn}[1]{\textnormal{#1}}
\renewcommand{\i}{{\mathrm{i}}}

\def\Ad{\ensuremath{\textnormal{Ad}}}

\def\g{\ensuremath{\mathfrak{g}}}
\def\k{\ensuremath{\mathfrak{k}}}
\def\t{\ensuremath{\mathfrak{t}}}

\def\n{\ensuremath{\mathfrak{n}}}

\def\so{\ensuremath{\mathfrak{so}}}

\def\spin{\ensuremath{\mathfrak{spin}}}

\def\hvee{\ensuremath{\textnormal{h}^\vee}}

\def\L{\ensuremath{\mathcal{L}}}

\def\R{\ensuremath{\mathcal{R}}}

\def\E{\ensuremath{\mathcal{E}}}
\def\B{\ensuremath{\mathcal{B}}}

\def\P{\ensuremath{\mathcal{P}}}

\def\Y{\ensuremath{\mathcal{Y}}}

\def\V{\ensuremath{\mathcal{V}}}
\def\W{\ensuremath{\mathcal{W}}}

\def\K{\ensuremath{\mathcal{K}}}
\def\X{\ensuremath{\mathcal{X}}}

\def\M{\ensuremath{\mathcal{M}}}

\def\Q{\ensuremath{\mathcal{Q}}}
\def\calK{\ensuremath{\mathcal{K}}}

\def\c{\ensuremath{\mathsf{c}}}

\def\dirac{\ensuremath{\slashed{\partial}}} 

\def\sY{\ensuremath{\mathsmaller{\mathcal{Y}}}}

\def\sDM{\ensuremath{\mathsmaller{DM}}}

\def\bC{\ensuremath{\mathbb{C}}}
\def\bR{\ensuremath{\mathbb{R}}}
\def\bZ{\ensuremath{\mathbb{Z}}}
\def\bN{\ensuremath{\mathbb{N}}}

\def\bB{\ensuremath{\mathbb{B}}}

\def\End{\ensuremath{\textnormal{End}}}

\def\Hom{\ensuremath{\textnormal{Hom}}}

\def\pr{\ensuremath{\textnormal{pr}}}

\def\.{\ensuremath{\cdot}}
\def\ker{\ensuremath{\textnormal{ker}}}
\def\coker{\ensuremath{\textnormal{coker}}}

\def\id{\ensuremath{\textnormal{id}}}

\def\Cliff{\ensuremath{\textnormal{Cliff}}}
\def\index{\ensuremath{\textnormal{index}}}

\def\dim{\ensuremath{\textnormal{dim}}}

\def\aff{\ensuremath{\textnormal{aff}}}
\def\KK{\ensuremath{\textnormal{KK}}}

\def\Sym{\ensuremath{\textnormal{Sym}}}

\def\Crit{\ensuremath{\textnormal{Crit}}}

\def\dom{\ensuremath{\textnormal{dom}}}

\def\K{\ensuremath{\textnormal{K}}}
\def\Bott{\ensuremath{\textnormal{Bott}}}
\def\Cyl{\ensuremath{\textnormal{Cyl}}}

\def\wt{\ensuremath{\textnormal{wt}}}

\def\Irr{\ensuremath{\tn{Irr}}}

\def\hLambda{\ensuremath{\widehat{\Lambda}}}

\begin{document}
\sloppy
\title[Norm-square localization]{Norm-square localization and the quantization of Hamiltonian loop~group~spaces}
\author{Yiannis Loizides} \address{Department of Mathematics, Pennsylvania State University, McAllister Building, University Park, Pennsylvania, 16802} \email{yxl649@psu.edu}
\author{Yanli Song} \address{Department of Mathematics, Washington University in St. Louis, Cupples Building, One Brookings Dr., St. Louis, Missouri, 63130} \email{yanlisong@wustl.edu}

\begin{abstract}
In an earlier article we introduced a new definition for the `quantization' of a Hamiltonian loop group space $\M$, involving the equivariant $L^2$-index of a Dirac-type operator $\scr{D}$ on a non-compact finite dimensional submanifold $\Y$ of $\M$.  In this article we study a deformation of this operator, similar to the work of Tian-Zhang and Ma-Zhang.  We obtain a formula for the index with infinitely many non-trivial contributions, indexed by the components of the critical set of the norm-square of the moment map.  This is the main part of a new proof of the $[Q,R]=0$ theorem for Hamiltonian loop group spaces.
\end{abstract}
\maketitle
\vspace{-0.5cm}

\section{Introduction}
Let $G$ be a compact connected Lie group with Lie algebra $\g$.  Let $\mu \colon M \rightarrow \g^\ast$ be a compact Hamiltonian $G$-space, equipped with a prequantum line bundle $L$.  Choosing a $G$-equivariant compatible almost complex structure, one obtains a spin-c Dirac operator $\dirac$ acting on sections of $\wedge T^\ast_{0,1}M \otimes L$.  Its equivariant index is an element of the representation ring of $G$.  The quantization-commutes-with-reduction ($[Q,R]=0$) theorem (cf. \cite{GuilleminSternbergConjecture,MeinrenkenSymplecticSurgery,
TianZhang,ParadanRiemannRoch}) says that the multiplicity of the irreducible representation with highest weight $\lambda$ equals the index of a similarly-defined operator on the symplectic quotient $\mu^{-1}(\lambda)/G_\lambda$ (this must be modified slightly in case $\lambda$ is not a regular value \cite{MeinrenkenSjamaar}).

Choose an invariant inner product on $\g$.  A well-known approach to $[Q,R]=0$ due to Tian-Zhang \cite{TianZhang} utilizes a deformation of $\dirac$:
\begin{equation} 
\label{eqn:TianZhangDef}
\dirac_t=\dirac -\i t\c(v_M), \qquad t \in \bR
\end{equation}
where $\c(-)$ denotes Clifford multiplication, and $v_M$ is the Hamiltonian vector field of the function $\tfrac{1}{2}\|\mu\|^2$.  As the parameter $t \rightarrow \infty$, sections in the kernel of $\dirac_t$ `localize' near the vanishing locus $Z=\{m \in M|v_M(m)=0\}=\Crit(\|\mu\|^2)$.  This turns out to be closely related to a formula of Paradan \cite{ParadanRiemannRoch} for the index of $\dirac$, involving contributions from the components of $Z$.  One has (cf. \cite{Kirwan})
\[ Z=\bigcup_{\beta \in \B} G\cdot(M^\beta \cap \mu^{-1}(\beta)),\]
where $\B \subset \t_+$ is a finite discrete subset of a positive Weyl chamber, and $M^\beta$ is the fixed-point submanifold of the 1-parameter subgroup generated by $\beta \in \t$.  Thus there will be a contribution from $Z_0=\mu^{-1}(0)$, together with `correction terms' from $0\ne \beta \in \B$.  One can show that the only contribution to the multiplicity of the trivial representation comes from $Z_0$.  Combined with an argument of a more local nature (near $\mu^{-1}(0)$), this leads to a proof of the $[Q,R]=0$ theorem.

In this article we prove analogous results for Hamiltonian loop group spaces.  This work builds on earlier articles \cite{LMSspinor} (joint with E. Meinrenken) and \cite{LSQuantLG}.  We very briefly summarize some results from these papers here, and in somewhat greater detail in Section \ref{sec:DiracTypeLoopGroup}.  Assume $G$ is also simple and simply connected for simplicity.  Let $LG$ denote the loop group of $G$, the space of maps $S^1 \rightarrow G$ of some fixed Sobolev level $s>\tfrac{1}{2}$.  Let $\Phi_{\M}\colon \M \rightarrow L\g^\ast$ be a Hamiltonian $LG$-space, with level $k>0$ prequantum line bundle $L$ (cf. \cite{AMWVerlinde}).  Fix a maximal torus $T \subset G$.  In earlier work \cite{LMSspinor} with E. Meinrenken, we constructed a finite-dimensional `global transversal' $\Y \subset \M$, as well as a canonical spinor module $S_0 \rightarrow \Y$.  The submanifold $\Y$ is a small `thickening' of the (possibly) singular subset $\X=\Phi_{\M}^{-1}(\t^\ast) \subset \M$.  In \cite{LSQuantLG} we studied a Dirac-type operator $\scr{D}$ on $\Y$ acting on sections of $\wedge \n_- \wh{\otimes}S$, where $S=S_0\otimes L$ and $\n_- \subset \g_{\bC}$ denotes the sum of the negative root spaces of $G$.  The operator $\scr{D}$ was shown to represent an index pairing between a spin-c Dirac operator for $S$, and the pull back of a Bott-Thom element for $\g/\t$, the latter formally playing the role of a Poincar{\'e} dual to $\X$ in $\Y$.

In \cite{LSQuantLG} we proved that $\scr{D}$ has a well-defined index $\index(\scr{D}) \in R^{-\infty}(T)$, which is, moreover, the Weyl-Kac numerator (restricted to $1 \in S^1_{\tn{rot}}$) of an element of the level $k$ fusion ring $R_k(G)$, the analogue of the representation ring for level $k$ positive energy representations of the loop group.  The latter motivated us to define the `quantization' of $(\M,L)$ as this particular element of $R_k(G)$.  In a related paper \cite{LoizidesGeomKHom}, the first author showed that this definition agrees with that of E. Meinrenken \cite{MeinrenkenKHomology} based on quasi-Hamiltonian spaces, twisted K-homology, and the Freed-Hopkins-Teleman theorem.

The manifold $\Y$ is equipped with a moment map $\phi \colon \Y \rightarrow \t$, which is proper on the support of the Bott-Thom element.  In Section \ref{sec:WittenDef} we introduce a deformation $\scr{D}_t$, defined by a formula similar to \eqref{eqn:TianZhangDef}, except that we use a re-scaled version of $v_M$ which has bounded norm.  For $t>0$, the operators $\scr{D}_t$ have the same index (in $R^{-\infty}(T)$) as $\scr{D}$.

In Section \ref{sec:MaZhangType} we prove a formula for $\index(\scr{D})$ which is inspired by work of Ma and Zhang \cite{MaZhangTransEll}.  The index is expressed as a sum of contributions indexed by the components $Z_\beta=\Y^\beta \cap \phi^{-1}(\beta)$ of $Z=\{v_{\sY}=0\}$:
\begin{equation}
\label{eqn:LimAPS0}
\index(\scr{D})=\sum_{\beta \in W\cdot \B} \lim_{t \rightarrow \infty}\index_{\tn{APS}}(\scr{D}_t\upharpoonright U_\beta).
\end{equation} 
Each contribution is a limit (in $R^{-\infty}(T)$) as $t \rightarrow \infty$, of the index of an Atiyah-Patodi-Singer boundary value problem on a compact neighborhood $U_\beta$ of $Z_\beta \cap \X$.  The sum over $\beta$ in \eqref{eqn:LimAPS0} is infinite, converging in $R^{-\infty}(T)$.  For the reader's benefit we have provided a brief introduction to elliptic boundary value problems in Section \ref{sec:PrelimDirac}, mostly following the recent references by B{\"a}r and Ballmann \cite{BarBallmannGuide,BarBallmann}.

In Section \ref{sec:Paradan} we follow a strategy similar to Ma-Zhang \cite{MaZhangBravermanIndex,MaZhangTransEll} and Braverman \cite{Braverman2002} to prove a formula for the contributions in \eqref{eqn:LimAPS0} in terms of transversally elliptic operators.  The end result is a formula in the spirit of Paradan \cite{ParadanRiemannRoch}:
\begin{equation} 
\label{eqn:NormSqrFormula}
\index(\scr{D})=\sum_{\beta \in W\cdot \B} \index(\sigma_{\beta,\theta} \otimes \Sym(\nu_\beta)),
\end{equation}
where $\sigma_{\beta,\theta}$ is a transversally elliptic symbol on the fixed-point set $\Y^\beta$, and $\nu_\beta$ is the normal bundle to $\Y^\beta$ in $\Y$ equipped with a `$\beta$-polarized' complex structure.  The formula \eqref{eqn:NormSqrFormula} is sometimes called a `norm-square localization' formula (or sometimes `non-abelian localization' formula), because the set of non-trivial contributions are indexed by the components of the critical set of the norm-square of the moment map $\|\Phi_{\M}\|^2$.

We remark that for a non-compact prequantized Hamiltonian $G$-space with proper moment map, the analogue of $\index(\scr{D})$ is not defined in general.  In their proof of the Vergne conjecture, Ma-Zhang \cite{MaZhangVergneConj,MaZhangTransEll} showed that, nevertheless, the right-hand-side of \eqref{eqn:LimAPS0} \emph{is} well-defined.  The resulting `quantization' of $M$ satisfies the $[Q,R]=0$ Theorem and behaves functorially under restriction to subgroups.  Thus one main difference in our setting is that for us, the global object $\index(\scr{D})$ is defined, and \eqref{eqn:LimAPS0} becomes a theorem.  Similar comments apply to the result of Paradan in \cite{ParadanFormalII}, or of Hochs and the second author in \cite{HochsSongSpinc}.  Another difference in our setting is the presence of the Bott-Thom element; this changes little from a conceptual point of view, although it complicates the proofs.

As in the work of Paradan \cite{ParadanRiemannRoch} on compact Hamiltonian $G$-spaces, the norm-square localization formula \eqref{eqn:NormSqrFormula} leads to a new proof of the $[Q,R]=0$ Theorem for Hamiltonian loop group spaces, and we discuss this briefly in Section \ref{ssec:QRrem}.  We do not present a complete proof of the $[Q,R]=0$ theorem here, as a couple of aspects would take us too far from our main focus.  These include an inequality involving the data of the affine Lie algebra $\wh{L\g}$ and a slightly more refined local description of the spin-c structure $S_0$ on $\Y$.  However the $[Q,R]=0$ theorem follows from our main theorem, together with a relatively small part of \cite{YiannisThesis} (or \cite{LMVerlindeQR} in preparation).  Perhaps the most important application of the $[Q,R]=0$ theorem for Hamiltonian loop group spaces is to the Verlinde formula; for context, see for example Bismut-Labourie \cite{BismutLabourie} for a symplectic approach to the Verlinde formulas, or the articles by Meinrenken \cite{LecturesGroupValued, MeinrenkenKHomology} for the relationship with $[Q,R]=0$ for Hamiltonian loop group spaces and quasi-Hamiltonian $G$-spaces. 

The PhD thesis of the first author \cite{YiannisThesis} (and the article \cite{LMVerlindeQR} in preparation) give a very different proof of a version of \eqref{eqn:NormSqrFormula} (the contributions are expressed rather differently), using combinatorial methods similar to Szenes and Vergne \cite{SzenesVergne2010}.  These references also contain some simple examples of \eqref{eqn:NormSqrFormula} for $G=SU(2), SU(3)$.

\vspace{0.3cm}

\noindent \textbf{Acknowledgements.} We thank Eckhard Meinrenken and Nigel Higson for helpful discussions and encouragement.  We thank the referee for their careful reading of the manuscript, and in particular for pointing out how to fix a gap in the proof of Proposition \ref{prop:bravindex} in an earlier draft. Y. Song is supported by NSF grant 1800667.

\vspace{0.3cm}

\noindent \textbf{Notation.}
We often use the summation convention (repeated indices are summed over).  If $V$ is a $\bZ_2$-graded vector bundle, then $V^+$ (resp. $V^-$) denotes the even (resp. odd) graded subbundle.  If $\st{D} \colon C^\infty(M,V) \rightarrow C^\infty(M,V)$ is an odd linear operator, then $\st{D}^{\pm}$ denote the induced maps $C^\infty(M,V^{\pm}) \rightarrow C^\infty(M,V^{\mp})$.  Unless stated otherwise, $[a,b]$ will denote the \emph{graded commutator} of the linear operators $a$, $b$.  We use the Koszul sign rule for graded tensor products.

On a Riemannian manifold $(M,g)$, $g$ will often be used to identify $TM \simeq T^\ast M$.  For a Hermitian vector bundle over $M$, the point-wise inner product (resp. norm) will be denoted $\pair{-}{-}$ (resp. $|-|$), while the inner product (resp. norm) on the space of $L^2$ sections for the Riemannian measure will be denoted $(-,-)_M$ (resp. $\|-\|_M$).  We will drop the subscript $M$ when the underlying manifold is clear from the context. 

If $K$ is a compact Lie group with Lie algebra $\k$, we write $\Irr(K)$ for the set of isomorphism classes of irreducible representations of $K$, and $R(K)$ for the representation ring.  The formal completion $R^{-\infty}(K)=\bZ^{\Irr(K)}$ of $R(K)$ consists of formal infinite linear combinations of irreducible representations $\pi \in \Irr(K)$ with coefficients in $\bZ$.  A sequence of elements in $R^{-\infty}(K)$ converges iff the coefficient of $\pi$ converges, for each $\pi \in \Irr(K)$.  Given a representation $W$ of $K$ and $\pi \in \Irr(K)$, $W_\pi \simeq \pi \otimes \Hom_K(\pi,W)$ denotes the $\pi$-isotypical subspace.  Given a $K$-equivariant linear map $A \colon W \rightarrow V$, $A_\pi$ denotes the map $W_\pi \rightarrow V_\pi$ obtained by restriction.  

If $K$ acts smoothly on a manifold $M$, $E$ is a $K$-equivariant vector bundle, and $\xi \in \k$, then $\L^E_\xi \colon C^\infty(M,E) \rightarrow C^\infty(M,E)$ denotes the differential operator obtained by differentiating the action of $\exp(t\xi)$ on $C^\infty(M,E)$ at $t=0$.  In case $E=M \times \bR$, $\L^E_\xi$ is a vector field on $M$ that we denote $\xi_M$.

Throughout $G$ will denote a fixed compact connected Lie group with Lie algebra $\g$.  Fix a choice of maximal torus $T$ with Lie algebra $\t$.  The normalizer of $T$ in $G$ is denoted $N_G(T)$, and the Weyl group $W=N_G(T)/T$.  The integral lattice is $\Lambda=\ker(\exp \colon \t \rightarrow T)$ and its dual $\Lambda^\ast=\Hom(\Lambda,\bZ)\simeq \Irr(T)$ is the (real) weight lattice.  Fix an invariant inner product $B$ on $\g$, which we use to identify $\g \simeq \g^\ast$. 

\section{Preliminaries on Dirac-type operators}\label{sec:PrelimDirac}
In this section we briefly recall some definitions and results on elliptic boundary value problems, mostly following \cite{BarBallmannGuide,BarBallmann}.  We also recall criteria for determining that a Dirac-type operator on a non-compact manifold is ($K$-)Fredholm, or has compact resolvent.

\subsection{Elliptic boundary value problems.}\label{sec:EllBd}
Let $M$ be a complete Riemannian manifold with compact boundary and interior unit normal vector $\nu$ along $\partial M$.  Let $E$, $F$ be Hermitian vector bundles, and let $\scr{D} \colon C^\infty(M,E) \rightarrow C^\infty(M,F)$ be a first-order differential operator.  The \emph{symbol} of $\scr{D}$ is the bundle map $\sigma_{\scr{D}} \colon T^\ast M\simeq TM \rightarrow \Hom(E,F)$ defined by\footnote{It is common to include a factor of $\sqrt{-1}$ in the definition.  We are following the convention in \cite{BarBallmannGuide}.}
\[ \sigma_{\scr{D}}(df)=[\scr{D},f].\]

Following B{\"a}r and Ballmann \cite[Section 2.1]{BarBallmannGuide}, we will say that $\scr{D}$ is of \emph{Dirac type} if its principal symbol satisfies the Clifford relations
\begin{align}
\label{eqn:cliffrel}
\sigma_{\scr{D}}(\xi)^\ast \sigma_{\scr{D}}(\eta)+\sigma_{\scr{D}}(\eta)^\ast \sigma_{\scr{D}}(\xi)&=2\pair{\xi}{\eta}\id_{E}\\
\nonumber
\sigma_{\scr{D}}(\xi) \sigma_{\scr{D}}(\eta)^\ast+\sigma_{\scr{D}}(\eta) \sigma_{\scr{D}}(\xi)^\ast &=2\pair{\xi}{\eta}\id_{F}
\end{align}
for all $\xi, \eta \in T^\ast M$.  This definition implies $\scr{D}$ is elliptic, and moreover that the composition 
\begin{equation} 
\label{eqn:AdaptedSymbol}
\sigma_{\scr{D}}(\nu)^{-1}\sigma_{\scr{D}}(\xi) \in \End(E|_{\partial M}), \qquad \xi \in T^\ast(\partial M)
\end{equation}
is skew-Hermitian (here $T^\ast(\partial M)$ is identified with the annihilator of $\nu$).  A first-order essentially self-adjoint differential operator $A \colon C^\infty(\partial M, E) \rightarrow C^\infty(\partial M, E)$ with principal symbol \eqref{eqn:AdaptedSymbol} is called an \emph{adapted boundary operator} for $\scr{D}$.  Since $\partial M$ is compact and $A$ is elliptic, $A$ has compact resolvent and there is an orthonormal basis of $L^2(\partial M,E)$ consisting of smooth eigensections of $A$.

Let $C^\infty_{cc}(M,E)$ denote the space of smooth sections of $E$ with compact support in the interior of $M$.  The \emph{formal adjoint} of $\scr{D}$ is the unique first-order differential operator $\scr{D}^\ast \colon C^\infty(M,F)\rightarrow C^\infty(M,E)$ such that
\[ (\scr{D}w,v)_{L^2(M,F)}=(w,\scr{D}^\ast v)_{L^2(M,E)} \]
for all $w \in C^\infty_{cc}(M,E)$, $v \in C^\infty_{cc}(M,F)$.  Then $\sigma_{\scr{D}^\ast}(\xi)=-\sigma_{\scr{D}}(\xi)^\ast$.  For compactly supported sections $w \in C^\infty_c(M,E)$, $v \in C^\infty_c(M,F)$ (which may be non-zero on the boundary), integration by parts gives \emph{Green's formula}:
\begin{equation}
\label{eqn:Green}
(\scr{D}w,v)_{L^2(M,F)}=(w,\scr{D}^\ast v)_{L^2(M,E)}-(\sigma_{\scr{D}}(\nu)w,v)_{L^2(\partial M,F)}.
\end{equation}

The \emph{minimal extension} $\scr{D}_{\tn{min}}$ is the closure (in the sense of unbounded operators on Hilbert spaces) of the operator with domain $C^\infty_{cc}(M,E)$.  The \emph{maximal extension} $\scr{D}_{\tn{max}}$ has domain the space of $L^2$-sections $e \in L^2(M,E)$ such that there exists a section $f \in L^2(M,F)$ with $\scr{D}e=f$ in the sense of distributions ($C^\infty_{cc}(M,F)$ being the space of test sections).  

Let $L^2_{\tn{loc}}(M,E)$ denote the space of sections of $E$ that are locally square-integrable (modulo sections vanishing almost everywhere), and let $H^1_{\st{loc}}(M,E)$ denote the space of sections $e \in L^2_{\tn{loc}}(M,E)$ whose weak first covariant derivatives lie in $L^2_{\tn{loc}}(M,E)$ as well.  If $M$ is compact then $L^2_{\tn{loc}}(M,E)=L^2(M,E)$ and $H^1_{\tn{loc}}(M,E)=H^1(M,E)$ is the usual Sobolev space.  It is convenient to introduce the smaller domain 
\[ H^1_{\scr{D}}(M,E)=H^1_{\st{loc}}(M,E)\cap \dom(\scr{D}_{\tn{max}}).\]
On a manifold with non-empty boundary $\dom(\scr{D}_{\tn{min}}) \subsetneq H^1_{\scr{D}}(M,E) \subsetneq \dom(\scr{D}_{\tn{max}})$, see B{\"a}r and Ballmann \cite{BarBallmann} for a precise characterization and detailed discussion.  By imposing suitable boundary conditions one obtains extensions in between the minimal and maximal extensions.  

In this article we will only need \emph{Atiyah-Patodi-Singer boundary conditions} \cite{AtiyahPatodiSingerI}, which are defined as follows.  Let 
\[ \scr{R}_{\partial M} \colon H^1_{\tn{loc}}(M,E) \rightarrow H^{1/2}_{\tn{loc}}(\partial M,E) \]
denote the \emph{trace map}, the continuous extension of the map given on smooth sections by restriction to the boundary.  Given an adapted boundary operator $A$, let $B_{<0}(A) \subset H^{1/2}(\partial M,E)$ denote the closure of the subspace generated by the negative eigenspaces of $A$.  The Atiyah-Patodi-Singer (APS) boundary value problem $(\scr{D},B_{<0}(A))$ is the extension of $\scr{D}$ with domain
\[ \dom(\scr{D},B_{<0}(A))=\{v \in H^1_{\scr{D}}(M,E)|\scr{R}_{\partial M}v \in B_{<0}(A)\}.\]
A short calculation using \eqref{eqn:cliffrel}, \eqref{eqn:AdaptedSymbol} shows that
\[ A^\vee=-\sigma_{\scr{D}}(\nu)\circ A \circ \sigma_{\scr{D}}(\nu)^{-1} \colon C^\infty(\partial M,F) \rightarrow C^\infty(\partial M,F),\]
is an adapted boundary operator for $\scr{D}^\ast$.  Let $B_{\le 0}(A^\vee) \subset H^{1/2}(\partial M,F)$ denote the closure of the subspace generated by the non-positive eigenspaces for $A^\vee$.  Using Green's formula, the Hilbert space adjoint of the operator $(\scr{D},B_{<0}(A))$ is the extension $(\scr{D}^\ast,B_{\le 0}(A^\vee))$ of $\scr{D}^\ast$ with domain
\[ \dom(\scr{D}^\ast,B_{\le 0}(A^\vee))=\{v \in H^1_{\scr{D}^\ast}(M,F)|\scr{R}_{\partial M}v \in B_{\le 0}(A^\vee) \}.\]

\subsection{Fredholm conditions and the splitting theorem.}
In this section we continue to assume that $\scr{D}$ is a Dirac-type operator on a complete Riemannian manifold with compact boundary.  Following B{\"a}r and Ballmann, we say that $\scr{D}$ is \emph{coercive at infinity} if there is a compact subset $S \subset M$ and constant $c>0$ such that
\begin{equation} 
\label{eqn:Coercive}
c\|v\|_{L^2(M,E)} \le \|\scr{D}v\|_{L^2(M,F)}
\end{equation}
for all $v\in C^\infty_c(M \setminus S,E)$.  More generally, suppose a compact Lie group $K$ acts on $M$, $E$, $F$ preserving the metrics and $\scr{D}$ is $K$-equivariant.  For $\pi \in \Irr(K)$, we say that $\scr{D}$ is $(K,\pi)$-\emph{coercive at infinity} if $\scr{D}_\pi$ is coercive, i.e. if there is a compact $S_\pi \subset M$ and constant $c_\pi$ such that \eqref{eqn:Coercive} holds for $v \in C^\infty_c(M\setminus S_\pi,E)_\pi$.  We say that $\scr{D}$ is $K$-\emph{coercive} if $\scr{D}$ is $(K,\pi)$-coercive for each $\pi \in \Irr(K)$.

A $K$-equivariant bounded linear operator $A \colon H \rightarrow H^\prime$ is $K$-\emph{Fredholm} if the operator $A_\pi \colon H_\pi \rightarrow H^\prime_\pi$ is Fredholm for each $\pi \in \Irr(K)$.  Such an operator has a $K$-\emph{index} in $R^{-\infty}(K)$, defined as
\[ \index(A)=\sum_{\pi \in \Irr(K)} \index(A_\pi)\pi.\]
\begin{remark}
\label{rem:Z2Graded}
We will also use the notation $\index(-)$ in closely related situations, and it should be clear from the context which interpretation is being used.  If $H$ is $\bZ_2$-graded and $A$ is an odd, possibly unbounded self-adjoint Fredholm operator, then $\index(A)$ denotes the index of $A^+$ viewed as a bounded Fredholm operator from $\dom(A^+)$ equipped with the graph norm, to $H$.  This coincides with $\dim(\ker(A^+))-\dim(\ker (A^-))$, the `graded dimension' of $\ker(A)$. 
\end{remark}

In the non-equivariant case, the following result is Theorem 8.5 in \cite{BarBallmann}, and the proof given there generalizes immediately to the equivariant case.
\begin{theorem}
Let $M$ be a complete Riemannian manifold with compact boundary and $\scr{D} \colon C^\infty(M,E) \rightarrow C^\infty(M,F)$ a Dirac-type operator.  Let $K$ be a compact Lie group acting on $M$, $E$, $F$ preserving the metrics.  Suppose $\scr{D}$ is $K$-equivariant and $A$ is a $K$-equivariant adapted boundary operator.  If $\scr{D}$, $\scr{D}^\ast$ are $K$-coercive then $(\scr{D},B_{<0}(A))$ is a $K$-Fredholm operator.
\end{theorem}

We are now ready to state the splitting theorem for the special case of APS boundary conditions.  In the non-equivariant case the statement below is Theorem 8.17 in \cite{BarBallmann}, and the proof given there generalizes easily to the equivariant case.  Splitting theorems of this general kind appear in many places in the literature. For but one other example (which allows even for families of operators), see Dai-Zhang \cite{dai1996splitting}.
\begin{theorem}
\label{thm:Split}
Let $M$ be a complete Riemannian manifold without boundary, and $\scr{D} \colon C^\infty(M,E) \rightarrow C^\infty(M,F)$ a Dirac-type operator.  Let $K$ be a compact Lie group acting on $M$, $E$, $F$ preserving the metrics, and assume $\scr{D}$, $\scr{D}^\ast$ are $K$-coercive.  Let $N \subset M$ be a compact $K$-invariant hypersurface with oriented normal bundle.  Cut $M$ along $N$ to obtain a manifold $M^\prime$ with $\partial M^\prime=N_1 \sqcup N_2$, where $N_1$, $N_2$ are two copies of $N$.  Let $\scr{D}^\prime$ denote the induced Dirac-type operator on $M^\prime$, acting between sections of the pullback bundles $E^\prime$, $F^\prime$.  Let $A$ be a $K$-equivariant adapted boundary operator for $\scr{D}^\prime$ along $N_1$; then $A \sqcup (-A)$ is an adapted boundary operator for $\scr{D}^\prime$.  Then we have the following equality in $R^{-\infty}(K)$:
\begin{equation} 
\label{eqn:SplitTheorem}
\index(\scr{D})=\index(\scr{D}^\prime,B_{<0}(A)\sqcup B_{\le 0}(-A)).
\end{equation}
\end{theorem}
\begin{remark}
In \eqref{eqn:SplitTheorem}, $(\scr{D}^\prime,B_{<0}(A)\sqcup B_{\le 0}(-A))$ denotes the extension of $\scr{D}^\prime$ with domain
\[ \{v \in H^1_{\scr{D}^\prime}(M^\prime,E^\prime)|\scr{R}_{N_1}v \in B_{<0}(A), \scr{R}_{N_2}v \in B_{\le 0}(-A) \}.\]
If the hypersurface $N$ is such that $N_1$, $N_2$ are contained in distinct components of $M^\prime$, then the right hand side of \eqref{eqn:SplitTheorem} becomes the sum of two indices on the two components of $M^\prime$.
\end{remark}

The following result on the discreteness of the spectrum for Schr{\"o}dinger-type operators is well-known, cf. \cite{ShubinSchrodinger, KucerovskyCallias}.  It is also closely related to the property of being `$\kappa$-coercive' for all $\kappa>0$ in \cite[Corollary 5.6]{BarBallmannGuide}.  We described the proof of a slightly more general result in \cite[Appendix B]{LSQuantLG}.
\begin{proposition}
\label{prop:DiscreteSpectrum}
Let $M$ be a complete Riemannian manifold without boundary.  Let $\st{D}$ be an essentially self-adjoint Dirac-type operator acting on sections of a Hermitian vector bundle $E$.  Let $V$ be a continuous function which is proper and bounded below.  Then $\st{D}^2+V$ is essentially self-adjoint with discrete spectrum.
\end{proposition}

For certain arguments later on it will be convenient to work with inequalities between semi-bounded operators.  If $A$ is a self-adjoint operator on a Hilbert space $H$ with domain $\dom(A)$ and spectrum in $[1,\infty)$, then one defines an associated positive definite quadratic form
\[ q_A(u_1,u_2)=(Au_1,u_2)\]
for all $u_1,u_2 \in \dom(A)$.  The completion of $\dom(A)$ using the inner product $q_A$ is a Hilbert space $\dom(q_A)$ which can be identified with $\dom(A^{1/2})$, and is known as the \emph{form domain} of $A$ (cf. \cite[VIII.6]{ReedSimonI}).  Given self-adjoint operators $A,B$ with spectrum in $[1,\infty)$ one writes
\[ A \ge B\]
if
\[ \dom(q_A) \subset \dom(q_B) \quad \text{and} \quad q_A(v,v)\ge q_B(v,v) \quad \forall v \in \dom(q_A)\]
(cf. \cite[XIII.2, p.85]{ReedSimonIV}).  Equivalently, $A \ge B$ if the inclusion mapping
\[ (\dom(q_A),q_A) \hookrightarrow (\dom(q_B),q_B)\]
is norm-decreasing.  It is enough to check that for each $v$ in a core for $A$, $v \in \dom(q_B)$ and $q_A(v,v)\ge q_B(v,v)$.  More generally if $A,B$ are self-adjoint operators with spectrum in $[-c,\infty)$ for some $c \ge 0$ then one writes $A \ge B$ if $A+c+1 \ge B+c+1$.

The following result is a consequence of Proposition \ref{prop:DiscreteSpectrum}, cf. \cite[Appendix B]{LSQuantLG} for details.
\begin{proposition}
\label{prop:CompactnessBoundedBelow}
Let $M$, $E$, $\st{D}$, $V$ be as in Proposition \ref{prop:DiscreteSpectrum}.  Let $K$ be a compact Lie group acting on $M$, $E$ preserving the metrics, and assume $\st{D}$, $V$ are $K$-equivariant.  Let $T$ be a $K$-equivariant self-adjoint operator on $L^2(M,E)$ with spectrum in $(0,\infty)$, and suppose for some $\pi \in \Irr(K)$ we have $T_\pi \ge (\st{D}^2+V)_\pi$.  Then $T_\pi^{-1}$ is a compact operator.
\end{proposition}
\begin{remark}
Later on we study a complicated Dirac-type operator $\scr{D}_t$ and will obtain an inequality as in the proposition (with $T=\scr{D}_t^2+1$ and $\st{D}$ a simpler Dirac-type operator) from a Bochner formula.  The estimate $(\scr{D}_t^2+1)_\pi \ge (\st{D}^2+V)_\pi$ implies $\scr{D}_t$ is $(K,\pi)$-coercive, and the proposition says moreover that $(\scr{D}_t)_\pi$ has discrete spectrum.
\end{remark}

\section{A Dirac-type operator associated to a loop group space}\label{sec:DiracTypeLoopGroup}
In this section we briefly review the setup and results from \cite{LMSspinor,LSQuantLG}, and begin to study a deformation of the operator under consideration.

\subsection{Hamiltonian loop group spaces.}\label{sec:LGSpace}
Let $LG$ denote the loops $S^1=\bR/\bZ \rightarrow G$ of some fixed Sobolev level $s > \tfrac{1}{2}$.  Point-wise multiplication of loops makes $LG$ into a Banach Lie group.  The Lie algebra of $LG$ is the space $L\g=\Omega^0(S^1,\g)$ consisting of loops in $\g$ of Sobolev class $s$.  We define the smooth dual $L\g^\ast$ to consist of $\g$-valued 1-forms on $S^1$ of Sobolev level $s-1$; the pairing between $L\g$, $L\g^\ast$ is given by the inner product, followed by integration over the circle.  $L\g^\ast$ is regarded as the space of connections on the trivial principal $G$-bundle over $S^1$, and carries a smooth, proper $LG$ action by gauge transformations:
\begin{equation}
\label{eqn:GaugeAction} 
g\cdot \xi=\Ad_g \xi - dg g^{-1}, \qquad g \in LG, \quad \xi \in L\g^\ast.
\end{equation}

The group $G$ embeds in $LG$ as the subgroup of constant loops.  The integral lattice $\Lambda$ of $G$ may be viewed as a subgroup of $LG$, by identifying $\lambda \in \Lambda$ with the closed geodesic $t\mapsto \exp(t\lambda)$.

\begin{definition}
A \emph{proper Hamiltonian $LG$-space} $(\M,\omega_{\M},\Phi_\M)$ is a Banach manifold $\M$ equipped with a smooth proper action of $LG$, a weakly non-degenerate, $LG$-invariant closed 2-form $\omega_\M$, and a smooth, proper, $LG$-equivariant map
\[ \Phi_{\M} \colon \M \rightarrow L\g^\ast \]
satisfying the moment map condition
\[ \iota(\xi_{\M})\omega_{\M}=-d\pair{\Phi_\M}{\xi}, \qquad \xi \in L\g.\]
\end{definition}
For a more detailed discussion of Hamiltonian loop group spaces, see for example \cite{MWVerlindeFactorization, AlekseevMalkinMeinrenken,BottTolmanWeitsman}.

\subsection{The global transversal $\Y$ of a Hamiltonian loop group space.}\label{sec:GlobTrans}
Let $\Phi_{\M} \colon \M \rightarrow L\g^\ast$ be a proper Hamiltonian $LG$-space.  The based loop group $\Omega G$ acts freely on $L\g^\ast$, and hence also on $\M$.  The quotient
\[ p \colon \M \rightarrow \M/\Omega G=:M \]
is a compact finite-dimensional $G$-manifold, and is an example of a \emph{quasi-Hamiltonian $G$-space} \cite{AlekseevMalkinMeinrenken}.  Since $L\g^\ast/\Omega G \simeq G$, $M$ comes equipped with a \emph{group-valued moment map}
\[ \Phi \colon M \rightarrow G.\]

Let $\st{B}_q(\g/\t)$ denote the ball of radius $q>0$ centred at the origin in $\g/\t$.  The normalizer $N_G(T)$ acts on $\st{B}_q(\g/\t)$ by the adjoint action.  Using the inner product there is an $N_G(T)$-equivariant identification $\g/\t \simeq \t^\perp$, where $\t^\perp$ is the orthogonal complement of $\t$ in $\g$.  There is an $N_G(T)$-equivariant map
\[ r_T \colon T \times \st{B}_q(\g/\t)=T \times \st{B}_q(\t^\perp) \rightarrow G, \qquad (t,\xi) \mapsto t\exp(\xi) \]
and for $q$ sufficiently small it is a diffeomorphism onto a tubular neighborhood $U$ of $T$ in $G$.  Define the $N_G(T)$-invariant open submanifold $Y$ of $M$ to be the pre-image:
\[ Y=\Phi^{-1}(U).\]
Let $\Y$ be the $\Lambda$-covering space of $Y$ defined as the fibre product $Y \times_U (\t \times \st{B}_q(\g/\t))$, using the map
\[ r_T \circ (\exp_T,\id) \colon \t \times \st{B}_q(\g/\t) \rightarrow U.\]  
Thus $Y=\Y/\Lambda$ and we have a pullback diagram  
\begin{equation} 
\label{eqn:Pullback}
\xymatrixcolsep{7pc}
\xymatrix{
\Y \ar[r]^{\Phi_{\Y}=(\phi,\phi^{\g/\t})} \ar[d]^{\pi} & \t \times \st{B}_q(\g/\t) \ar[d]^{r_T\circ (\exp_T,\id)} \\
Y \ar[r]_{\Phi|_Y} & U
}
\end{equation}
The first component $\phi$ of the map $\Phi_{\Y}$ defined by \eqref{eqn:Pullback} is a moment map for the $N_G(T)\ltimes \Lambda$-action (using $\t \simeq \t^\ast$), and $\Y$ can be seen to be a degenerate Hamiltonian $N_G(T) \ltimes \Lambda$-space.

Interestingly, $\Y$ can be embedded $N_G(T)\ltimes \Lambda$-equivariantly into the infinite dimensional manifold $\M$, such that
\begin{itemize}
\item the map obtained by composition $\Y \hookrightarrow \M \xrightarrow{p} M$ has image $Y$ and coincides with the covering map $\pi \colon \Y \rightarrow Y$;
\item the image of $\Y$ in $\M$ is a small `thickening' of the (possibly) singular closed subset
\[ \X=\Phi_{\M}^{-1}(\t) \]
where $\t \hookrightarrow L\g^\ast$ is embedded as constant connections;
\item the image of $\Y$ in $\M$ intersects all the $LG$-orbits transversally. 
\end{itemize}
In earlier work \cite[Section 6.4]{LMSspinor} with E. Meinrenken, we showed how to construct such an embedding, depending on the choice of a connection on the principal $\Omega G$-bundle $L\g^\ast \rightarrow G$.\footnote{In \cite{LMSspinor} we actually worked with a slightly larger space $\P G$ (a principal $G$-bundle over $L\g^\ast$), which was desirable for certain purposes, although we have avoided discussing it here for simplicity.  The embedding referred to here can be constructed by the same method.}  In \emph{loc. cit.} we referred to $\Y$ as a \emph{global transversal}.  One reason this perspective is useful is for the description of a canonical spin-c structure on $\Y$, explained in the same article.

\subsection{A $1^{st}$-order elliptic operator on $\Y$.}
\label{sec:FirstOrdY}
Redefining $Y$ to be smaller if needed, one can construct a complete $N_G(T)$-invariant Riemannian metric $g$ on $Y$, such that $Y$ has a cylindrical end
\[ \Cyl_Q = Q \times (1,\infty),\]
where $Q \subset Y$ is a compact $N_G(T)$-invariant hypersurface, the complement $Y \setminus \Cyl_Q$ is compact, and the metric
\[ g|_{\Cyl_Q}=dx^2+g_Q,\]
where $g_Q$ is a metric on $Q$ and $x \in (1,\infty)$.  This normal form can be constructed by beginning with a slightly larger open subset $Y^\prime=\Phi^{-1}(U^\prime)$, where $U^\prime=r_T(T\times \st{B}_{q^\prime}(\g/\t))$, and then choosing $Q$ to be the inverse image of a regular value $q$ of the map $|\pr_2 \circ r_T^{-1}\circ \Phi|_{Y^\prime}|\colon Y^\prime \rightarrow [0,q^\prime)$.  The metric $g$ can be constructed by patching together a $N_G(T)$-invariant metric on $Y^\prime$ with a cylindrical metric on a collaring neighborhood of $Q$ using a partition of unity; cf. \cite[Section 4.7.1]{LSQuantLG} for further details.  One can arrange that the vector field $\partial_x$ on $Y$ extends continuously by $0$ to a neighborhood of $Y$ in $M$.

The pullback of $g$ to $\Y$ is a $N_G(T)\ltimes \Lambda$-invariant complete metric on $\Y$.  The Riemannian volume determines a measure on $\Y$.  Let $\Q$, $\Cyl_{\Q}$ denote the inverse image under the quotient map $\Y \rightarrow Y$ of $Q$, $\Cyl_Q$ respectively.  For convenience extend $x \colon \Cyl_{\Q} \rightarrow (1,\infty)$ to a smooth $N_G(T)\ltimes \Lambda$-invariant function $x \colon \Y \rightarrow (0,\infty)$, such that $x^{-1}(1,\infty)=\Cyl_{\Q}$.

Let $E=E^+\oplus E^-$ be a $\bZ_2$-graded $T \times \Lambda$-equivariant Hermitian vector bundle over $\Y$ such that $E^{\pm}|_{\Q}$ are trivial.  Let $\theta \in C^\infty(\Y,\End(E))$ be a bounded, odd, self-adjoint $T\times \Lambda$-equivariant bundle endomorphism, such that $\theta^2=\id$ on $\Cyl_{\Q}$.  Let 
\[ f \colon [0,\infty)\rightarrow [1,\infty)\]
be a smooth, monotone non-decreasing function, equal to $1$ on a neighborhood of $[0,1]$, such that 
\begin{equation}
\label{eqn:fcnf}
f(s) \xrightarrow{s\rightarrow \infty} +\infty \qquad \text{and} \qquad f(s)^{-2}f^\prime(s) \xrightarrow{s\rightarrow \infty} 0.
\end{equation}
The composition $f\circ x$ will serve as a kind of potential function on $\Y$; to keep the notation from becoming overly cluttered we will continue to denote this composition just by $f$.

The product group $T \times \Lambda$ sits as a subgroup of $LG$.  Given a $U(1)$ central extension of $LG$, we obtain a central extension of $T \times \Lambda$ by restriction.  Any central extension is trivial over the torus hence is of the form $T\ltimes \hLambda$, for some $U(1)$ central extension $\hLambda$ of $\Lambda$.  In any case, let $T\ltimes \hLambda$ be a $U(1)$ central extension.  We assume that there is a homomorphism $\kappa \colon \Lambda \rightarrow \Lambda^\ast$ satisfying $\pair{\kappa(\lambda)}{\lambda}>0$ $\forall 0 \ne \lambda \in \Lambda$,\footnote{Such $\kappa$ are in one-one correspondence with inner products on $\t$ which take integer values on $\Lambda$.} and such that for all $t \in T$, $\wh{\lambda} \in \hLambda$ we have the following commutation relation:
\begin{equation}
\label{eqn:CommutationRelation}
\wh{\lambda}\,\, t\,\, \wh{\lambda}^{-1}\,\, t^{-1}=t^{-\kappa(\lambda)}.
\end{equation}
We comment on the role of this assumption in Remark \ref{rem:commrole} below. 

The following summarizes some results from \cite{LSQuantLG}.
\begin{theorem}
\label{thm:FredholmI}
Let $T\ltimes \hLambda$ be a central extension of $T \times \Lambda$ satisfying \eqref{eqn:CommutationRelation}.  Let $S$ be a $T\ltimes \hLambda$-equivariant spinor module on $\Y$, equipped with a Clifford connection.  With $E$ as above, let $\st{D}$ denote the corresponding Dirac operator acting on sections of $\E:=E\wh{\otimes} S$.  Let $\theta$, $f$ be as above with $f$ satisfying the growth conditions \eqref{eqn:fcnf}.  Then the Dirac-type operator
\begin{equation} 
\label{eqn:DefD}
\scr{D}=\st{D}+f\theta \wh{\otimes} 1
\end{equation}
is $T$-Fredholm.
\end{theorem}
We will usually write $f\theta$ instead of $f\theta \wh{\otimes}1$ when it should not cause confusion.  In terms of a local orthonormal frame $X_1,...,X_{\dim(\Y)}$, the Dirac operator $\st{D}$ is
\[ \st{D}(e\wh{\otimes}s)=(-1)^{\deg(e)}\big(\nabla^E_{X_n} e\wh{\otimes}\c(X_n)s+e\wh{\otimes}\c(X_n)\nabla^S_{X_n}s\big),\]
the $(-1)^{\deg(e)}$ appears because of the Koszul sign rule.

\begin{remark}[\emph{Examples}]
In \cite{LMSspinor} we constructed a canonical $\wh{LG}$-equivariant spinor module $S_0$ for the vector bundle $p^\ast TM$, where $p \colon \M \rightarrow M=\M/\Omega G$ is the quotient map. The relevant central extension of $LG$ here is the \emph{spin central extension} (cf. \cite{PressleySegal,FHTII}).  If $G$ is semisimple, the resulting central extension of $T \times \Lambda$ satisfies \eqref{eqn:CommutationRelation}. Above we mentioned that $\Y$ embeds into $\M$ as a finite-dimensional submanifold such that the map obtained by composition $\Y \hookrightarrow \M \xrightarrow{p} M$ has image $Y$ and coincides with the covering map $\pi \colon \Y \rightarrow Y$.  Since $Y$ is an open subset of $M$, it follows that the pullback of $S_0$ to $\Y$ is a $T \ltimes \wh{\Lambda}$-equivariant spinor module for $\Y$ satisfying \eqref{eqn:CommutationRelation}.  For the quantization problem, one also wants to consider spinor modules obtained from $S_0$ by twisting with auxiliary line bundles. 
\end{remark}

\begin{remark}[\emph{On the role of the commutation relation}]
\label{rem:commrole}
The commutation relation for the central extension that acts on the spinor module \eqref{eqn:CommutationRelation} plays a key role in \cite{LSQuantLG} in showing that the operator \eqref{eqn:DefD} is $T$-Fredholm.  We will see that it also has a role in the present article, stemming from Lemma \ref{lem:MuVPairing}. Looking ahead to Section \ref{sec:WittenDef}, we will study a deformation $\scr{D}_t$ of $\scr{D}$, and the commutation relation implies that one of the terms appearing in the Bochner-type formula for $\scr{D}_t^2$ goes to $\infty$ as one translates out to infinity in $\Y$ using the action of $\Lambda$.  This observation plays a key role in the analysis of the deformation $\scr{D}_t$ in Sections \ref{sec:WittenDef}, \ref{sec:MaZhangType}.
\end{remark}

\begin{remark}[\emph{On the index as a pairing in K-theory}]
\label{rem:KThyPair}
\begin{enumerate}
\item The pair $(E,\theta)$ descend to $Y$, and represent a class $[\theta] \in \K^0_T(Y)$.  In \cite{LSQuantLG}, we showed that the Dirac operator $\st{D}^S$ for $S \rightarrow \Y$ (set $E=\bC$) defines a class $[\st{D}^S]$ in the analytic K-homology group $\KK(T\ltimes C_0(Y),\bC)$, and the operator in Theorem \ref{thm:FredholmI} was interpreted as a representative for the Kasparov product $j_T([\theta])\otimes_{T\ltimes C_0(Y)}[\st{D}^S]$.
\item The map $\kappa$ determines an action of $\Lambda$ on $\Irr(T)\simeq \Lambda^\ast$ by translations.  The $\hLambda$-equivariance of $\scr{D}$, together with \eqref{eqn:CommutationRelation} imply that $\index(\scr{D})\in R^{-\infty}(T)=\bZ^{\Lambda^\ast}$ is invariant under this action.
\item An important special case is for $(E,\theta)$ representing the pullback under $\phi^{\g/\t}$ of the Bott element for $\g/\t \simeq \n_-$, in which case $E$ is trivial with fibres $\wedge \n_-$.  In this case there is additional anti-symmetry under the Weyl group, and $\index(\scr{D})$ is anti-symmetric under a suitable action of the affine Weyl group, see \cite{LSQuantLG} for details, and also Section \ref{ssec:QRrem} for further discussion.
\end{enumerate}
\end{remark}

\section{The deformation $\scr{D}_t$.}\label{sec:WittenDef}
We use the inner product to identify $\t$ with $\t^\ast$, hence $\phi$ can be viewed as a map $\Y \rightarrow \t$.  Choose a smooth bounded function
\[ \chi \colon [0,\infty) \rightarrow (0,\infty) \] 
such that, as $r \rightarrow \infty$, $r\chi(r^2)$ and $r\chi^\prime(r)$ remain bounded while $r\chi(r) \rightarrow \infty$.  For example, one can take
\[ \chi(r)=\frac{1}{\sqrt{1+r}}.\]
\begin{definition}
Given $\chi$ as above, the corresponding \emph{taming map} is
\[ v \colon \Y \rightarrow \t, \qquad v=\chi(|\phi|^2)\cdot \phi.\]
Note that $v$ is a bounded map by construction.  The vector field generated by $v$, denoted $v_{\sY}$, is defined by
\[ v_{\sY}|_y=v(y)_{\sY}|_y.\]
On the right-hand-side, $v(y) \in \t$ generates a vector field $v(y)_{\sY}$ on $\Y$ that is then evaluated at $y \in \Y$.
\end{definition}  
\begin{remark}
The terminology `taming map' was introduced by Braverman \cite{Braverman2002}.  In his application, the taming map was required to satisfy certain growth conditions at infinity (and would not be bounded).  The taming map here is closer to that used by Harada and Karshon \cite{KarshonHarada}.
\end{remark}
One checks that
\[ Z:=\{y \in \Y|v_{\sY}(y)=0\}=\bigcup_{\beta \in \t} \Y^\beta \cap \phi^{-1}(\beta).\]
The set of $\beta \in \t_+$ such that $\Y^\beta \cap \phi^{-1}(\beta)\ne\emptyset$ is a discrete subset $\B \subset \t_+$.  We refer to the subsets
\[ Z_\beta=\Y^\beta \cap \phi^{-1}(\beta), \qquad \beta \in W \cdot \B \]
as the `components' of the vanishing locus $Z$ (although they are not necessarily connected).  
\begin{remark}
Recall the (possibly singular) subset $\X=(\phi^{\g/\t})^{-1}(0) \subset \Y$.  If the tubular neighborhood $U \supset T$ is chosen sufficiently small, then $\Y^\beta \cap \phi^{-1}(\beta) \ne \emptyset$ $\Leftrightarrow$ $\X^\beta \cap \phi^{-1}(\beta) \ne \emptyset$.  Under the embedding $\Y \hookrightarrow \M$, $\X$ is identified with $\Phi_{\M}^{-1}(\t^\ast)$, hence $\X^\beta \cap \phi^{-1}(\beta)=\M^\beta \cap \Phi_{\M}^{-1}(\beta)$.  For $\beta \in \t_+$ we see that
\[ \Y^\beta \cap \phi^{-1}(\beta) \ne \emptyset \quad \Leftrightarrow G\cdot (\M^\beta \cap \Phi_{\M}^{-1}(\beta)) \ne \emptyset \]
and the latter subset of $\M$ is the component of the critical set of the norm-square of the moment map $\|\Phi_{\M}\|^2$ labelled by $\beta$, cf. \cite{DHNormSquare} for further discussion.
\end{remark}
\begin{definition}
The deformation $\scr{D}_t$ of $\scr{D}$ is the family of Dirac-type operators
\[ \scr{D}_t=\st{D}+(1+t)f\theta \wh{\otimes} 1-\i t\wh{\otimes}\c(v_{\sY}), \qquad t \in \bR.\]
\end{definition}
We will drop the `$1\wh{\otimes}$' when it should not cause confusion; in this notation one should remember that the operators $\c(v_{\sY})$, $\theta$ graded commute.

\subsection{Bochner formula for $\scr{D}_t^2$.}
Let $v_j$ denote the components of $v$ with respect to an orthonormal basis $\xi^j$, $j=1,...,\dim(\t)$ for $\t$; these are smooth bounded $\bR$-valued functions on $\Y$.  We refer to $\det(S)=\Hom_{\Cliff(T\Y)}(S^\ast,S)$ as the \emph{determinant line bundle} of the spin-c structure $S$.  The chosen connection on $S$ determines a connection on $\det(S)$ in the usual way.  We define the spin-c moment map $\mu \colon \Y \rightarrow \t^\ast$ by
\begin{equation} 
\label{eqn:SpincMoment}
2\pi \i \pair{\mu}{\xi}=\tfrac{1}{2}\Big(\L_\xi^{\det(S)}-\nabla_{\xi_{\sY}}^{\det(S)}\Big), \qquad \xi \in \t.
\end{equation}
(Note that the right-hand-side is an operator on $C^\infty(\Y,\det(S))$ which commutes with multiplication by functions, hence defines a section of $\End(\det(S))=\Y \times \bC$.)  Similarly using the chosen connection on $E$ we define a moment map (cf. \cite{BerlineGetzlerVergne}) $\mu_E \in \t^\ast \otimes C^\infty(\Y,\End(E))$ by
\begin{equation}
\label{eqn:EMoment}
2\pi \i \pair{\mu_E}{\xi}=\L^E_\xi-\nabla^E_{\xi_{\sY}}, \qquad \xi \in \t.
\end{equation}
Using the metric we identify $T\Y \simeq T^\ast \Y$.  If $R$ is a Killing vector field and $\nabla$ the Levi-Civita connection, the bundle endomorphism $\nabla_\bullet R \colon X \mapsto \nabla_X R$ is skew-adjoint with respect to the metric on $T\Y$, hence defines a section of the adjoint bundle $\so(T\Y)$.  The latter is identified with a subbundle of $\Cliff(T\Y)$ (recall $\so(V)\simeq \spin(V) \subset \Cliff(V)$ for a Euclidean vector space $V$), hence for $R$ Killing we obtain a section $\c(\nabla_\bullet R) \in \End(S)$.  In terms of a local orthonormal frame $X_1,...,X_{\dim(\Y)}$ we have
\[ \c(\nabla_\bullet R)=\tfrac{1}{4}\c(X_n)\c(\nabla_{X_n}R).\]
The connection $\nabla^E$ induces a connection $\nabla^{\End(E)}$ on $\End(E)$ satisfying $\nabla_X^{\End(E)}\sigma=[\nabla_X^E,\sigma]$ as sections of $\End(E)$, for all vector fields $X$ and $\sigma \in C^\infty(\Y,\End(E))$.  The covariant differential $\nabla^{\End(E)} \theta$ defines a section of $T^\ast \Y \otimes \End(E)$.  It is convenient to write $\c(\nabla^{\End(E)} \theta)$ for the section of $\End(E\wh{\otimes}S)$ obtained by applying the Clifford action $\c$ to the $T^\ast \Y\simeq T\Y$ part of the tensor $\nabla^{\End(E)}\theta$.  In terms of the local frame
\[ \c(\nabla^{\End(E)} \theta)=-\nabla^{\End(E)}_{X_n}\theta \wh{\otimes}\c(X_n) \in \End(E\wh{\otimes}S).\]
\begin{proposition}[Bochner formula, cf. \cite{TianZhang,HochsSongSpinc}]
\label{prop:Bochner}
Let $\xi^j$ be an orthonormal basis of $\t$.  The square of the deformed operator is given by
\begin{equation} 
\label{eqn:Bochner0}
\scr{D}_t^2=\st{D}_{\theta,t}^2+t^2|v_{\sY}|^2+4\pi t\pair{\mu+\mu_E}{v}+2\i t v_j\L_{\xi^j}^{\E}+\i t b 
\end{equation}
where
\[ b=-2v_j\c(\nabla_\bullet \xi^j_{\sY})-\c(dv_j)\c(\xi^j_{\sY}) \]
is a smooth bounded section of $\End(\E)$, $\st{D}_{\theta,t}=\st{D}+(1+t)f\theta$ and
\begin{equation}
\label{eqn:Bochner00}
\st{D}_{\theta,t}^2=\st{D}^2+(1+t)^2f^2\theta^2+(1+t)\c(df)\theta+(1+t)f\c(\nabla^{\End(E)} \theta).
\end{equation}
\end{proposition}
\begin{proof}
We have
\begin{equation} 
\label{eqn:Bochner1}
\scr{D}_t^2=\st{D}_{\theta,t}^2+t^2|v_{\sY}|^2-\i t[\st{D}_{\theta,t},\c(v_{\sY})].
\end{equation}
Since $\theta$, $\c(v_{\sY})$ graded anti-commute, the cross-term simplifies to
\[ [\st{D}_{\theta,t},\c(v_{\sY})]=[\st{D},\c(v_{\sY})].\]
Let $X_1,...,X_{\dim(\Y)}$ be a local orthonormal frame.  Then
\begin{align}
\label{eqn:Bochner2} 
[\st{D},\c(v_{\sY})]&=[\c(X_n),\c(v_{\sY})]\nabla^{\E}_{X_n}+\c(X_n)[\nabla^{\E}_{X_n},\c(v_{\sY})]\nonumber \\
&=-2v_j\nabla^{\E}_{\xi^j_{\sY}}+\c(X_n)\c(\nabla_{X_n}v_{\sY})\nonumber \\
&=-2v_j\nabla^{\E}_{\xi^j_{\sY}}+\c(X_n)\Big((\nabla_{X_n}v_j)\c(\xi^j_{\sY})+v_j\c(\nabla_{X_n}\xi^j_{\sY})\Big)\nonumber \\
&=-2v_j\nabla^{\E}_{\xi^j_{\sY}}+\c(d v_j)\c(\xi^j_{\sY})+4v_j\c(\nabla_\bullet \xi^j_{\sY})
\end{align}
and the expression in the last line holds globally.

Locally on $\Y$ we can choose a spin structure $S(T\Y)$ and square-root $\det(S)^{1/2}$ such that locally $S\simeq S(T\Y)\otimes \det(S)^{1/2}$.
The Levi-Civita connection is torsion-free, implying the following identity of operators acting on vector fields:
\[ \nabla_X=\L_X+\nabla_\bullet X.\]
If $X$ is Killing, it follows that
\begin{equation} 
\label{eqn:Bochner3}
\nabla_X^{S(T\Y)}=\L_X^{S(T\Y)}+\c(\nabla_\bullet X),
\end{equation}
where $\nabla^{S(T\Y)}$ is the spin connection. Using the definitions of $\mu$, $\mu_E$ we have
\begin{equation} 
\label{eqn:Bochner4}
\nabla^{\det(S)^{1/2}}_{\xi^j_{\sY}}=\L_{\xi^j}^{\det(S)^{1/2}}-2\pi \i\pair{\mu}{\xi^j}, \qquad \nabla^E_{\xi^j_{\sY}}=\L^E_{\xi^j}-2\pi \i\pair{\mu_E}{\xi^j}.
\end{equation}
Combining \eqref{eqn:Bochner3}, \eqref{eqn:Bochner4}
\begin{equation} 
\label{eqn:Bochner5}
\nabla^{\E}_{\xi^j_{\sY}}=\L^{\E}_{\xi^j}-2\pi \i\pair{\mu+\mu_E}{\xi^j}+\c(\nabla_\bullet \xi^j_{\sY}),
\end{equation}
and this expression holds globally.  Combining equations \eqref{eqn:Bochner1}, \eqref{eqn:Bochner2}, \eqref{eqn:Bochner5} gives \eqref{eqn:Bochner0}.

Using $\Lambda$-invariance and our assumption that the metric and action both take product forms on $\Cyl_{\Q}$, it follows that $\c(\xi^j_{\sY})$, $\c(\nabla_\bullet \xi^j_{\sY})$ are bounded operators.  Note that
\[ dv_j=2\chi^\prime(|\phi|^2)\phi_i\phi_j d\phi_i+\chi(|\phi|^2)d\phi_j. \]
The functions $\chi(|\phi|^2)$, $\chi^\prime(|\phi|^2)\phi_i \phi_j$ are bounded according to the conditions on $\chi$.  The 1-form $d\phi_j$ descends to $Y$.  To show that it has bounded norm on $Y$, it suffices to consider its behavior on $\Cyl_Q$.  Let $Q^\prime \subset Q$ be a small open subset such that we have a local orthonormal frame $X_2,...,X_{\dim(Y)}$ for $g_Q$ on $Q^\prime$.  Let $X_1=\partial_x$, so that $X_n$, $n=1,...,\dim(Y)$ is a local orthonormal frame for $g$ on $Y^\prime=Q^\prime \times (1,\infty) \subset Y$.  On $Q^\prime \times (1,\infty)$ we have
\[ g^\sharp(d\phi_j)=d\phi_j(X_n) X_n.\]
The 1-form $d\phi_j$ extends smoothly to a neighborhood of the closure $\ol{Y}$ of $Y$ in $M$, and the vector fields $X_n$ extend continuously to the closure of $Y^\prime$ in $M$ (in particular the vector field $\partial_x$ on $Y$ extends continuously by $0$ to $\ol{Y}$).  Since $\ol{Y} \subset M$ is compact it follows that $d\phi_j(X_n)$ is a bounded function on $Y^\prime$, hence $g^\sharp(d\phi_j)$ has bounded norm on $Y$.  This proves $b$ is bounded.

Equation \eqref{eqn:Bochner00} follows from
\[ [\st{D},f\theta]=[\st{D},f]\theta+f [\st{D},\theta]=\c(df)\theta+f \c(\nabla^{\End(E)} \theta).\]
\end{proof}

\subsection{Fredholm property for $\scr{D}_t$.}
The term $\pair{\mu}{v}$ in the Bochner formula will play a crucial role, owing to the following.
\begin{lemma}
\label{lem:MuVPairing}
Let $W \subset Y$ be a compact subset, and $\W=\pi^{-1}(W) \subset \Y$.  Then $\pair{\mu}{v}|_{\W}$ is proper and bounded below.  Consequently the sum $\pair{\mu}{v}+f$ is proper and bounded below on $\Y$.
\end{lemma}
\begin{proof}
For the first claim, since the quotient $\W/\Lambda=W$ is compact, it suffices to show that $\pair{\mu}{v}(\lambda.y)\xrightarrow{|\lambda|\rightarrow \infty} \infty$ for each $y \in \Y$.  The commutation relations \eqref{eqn:CommutationRelation} imply $\mu(\lambda.y)=\mu(y)+\kappa(\lambda)$.  On the other hand $\phi(\lambda.y)=\phi(y)+\lambda$.  Thus
\[ \pair{\mu}{v}(\lambda.y)=\chi(|\phi(y)+\lambda|^2)\Big(\pair{\mu(y)}{\phi(y)}+\pair{\mu(y)+\kappa^\ast \phi(y)}{\lambda}+\pair{\kappa(\lambda)}{\lambda}\Big).\]
Our assumptions on $\chi$ imply that the first two terms in the brackets, when multiplied by $\chi(|\phi(y)+\lambda|^2)$, remain bounded as $|\lambda|\rightarrow \infty$.  Since $\kappa$ was assumed to be positive definite, for large $|\lambda|$ the third term behaves like a constant times $\chi(r^2)r^2$, which goes to infinity as $r=|\lambda|$ goes to infinity, again by our assumptions on $\chi$.

For the second claim, consider the joint function $(\pair{\mu}{v},f)\colon \Y \rightarrow \bR^2$.  This map is proper, because for any compact subset $K \subset \bR$, $\W=f^{-1}(K)$ is a subset of the type considered above, on which $\pair{\mu}{v}$ is proper.  For any $c \in \bR$, the map $(s,t) \in [c,\infty)\times [c,\infty) \mapsto s+t \in \bR$ is proper.  Since $\pair{\mu}{v}$, $f$ are both bounded below, and as the composition of proper maps is proper, it follows that $\pair{\mu}{v}+f$ is proper.
\end{proof}

\begin{proposition}
\label{prop:CtsTFredholm}
For $t>0$, $\scr{D}_t$ is $T$-Fredholm.
\end{proposition}
\begin{remark}
The case $t=0$, stated in Theorem \ref{thm:FredholmI}, was proved in \cite{LSQuantLG} using somewhat different methods from the proof for $t>0$ below.  We will make use of some positivity in the proof, so it is not clear what happens when $t<0$.  Note that, in contrast to the space of bounded Fredholm operators, the space of bounded $T$-Fredholm operators is \emph{not} open in the norm topology, so it is possible that $\scr{D}_t$ is not $T$-Fredholm for any $t<0$.  
\end{remark}
\begin{proof}
Fix $t>0$ and $\lambda \in \Lambda^\ast \simeq \Irr(T)$.  We will show that $(\scr{D}_t^2+1)^{-1}_\lambda$ is compact hence $(\scr{D}_t)_\lambda$ has discrete spectrum, and a fortiori $(\scr{D}_t)_\lambda$ is Fredholm.  By Proposition \ref{prop:CompactnessBoundedBelow} it suffices to prove an inequality of the form
\begin{equation} 
\label{eqn:FredholmDeformation1}
(\scr{D}_t^2)_\lambda \ge (\st{D}^2+V)_\lambda
\end{equation}
for some $T$-invariant continuous potential $V$ which is proper and bounded below.  

We find a suitable $V$ using equation \eqref{eqn:Bochner0}.  Notice that $|v_{\sY}|$, $|\pair{\mu_E}{v}|$, $|\theta|$, $|\nabla^{\End(E)}\theta|$ are all bounded globally on $\Y$, using a combination of the facts that (1) $v$ is bounded, (2) the metric $g$, as well as the sections $\mu_E$, $\theta$ are $\Lambda$-invariant and are constant in the $x$-direction on $\Cyl_{\Q}$.  We conclude that equation \eqref{eqn:Bochner0} takes the form
\begin{equation}
\label{eqn:RoughBochner}
\scr{D}_t^2=\st{D}^2+t^2|v_{\sY}|^2+(1+t)^2f^2\theta^2+4\pi t\pair{\mu}{v}+(1+t)fb_1+(1+t)\c(df)b_2 +2\i t v_j\L_{\xi^j}^{\E}
\end{equation}
where $b_1$, $b_2$ are bounded (uniformly in $t$) sections of $\End(\E)$; note here we have taken advantage of the fact that $f \ge 1$ to hide the terms in \eqref{eqn:Bochner0} containing $b$, $\pair{\mu_E}{v}$ inside $b_1$.

As in Tian-Zhang \cite{TianZhang}, a key observation is that, restricted to the $\lambda$-isotypic component, the operators $\L^{\E}_{\xi^j}$ in \eqref{eqn:Bochner0} become \emph{bounded} (in fact since $T$ is abelian, they restrict to multiplication operators by a constant), so $2v_j\L^{\E}_{\xi^j}$ is bounded by a constant $c_\lambda$ (the supremum over $y \in \Y$ of $2|v_j(y)\cdot 2\pi \pair{\lambda}{\xi^j}|=4\pi |\pair{\lambda}{v(y)}|$).  

Since $\theta$ is self-adjoint, $\theta^2(y)$ is a non-negative endomorphism of $E$ for each $y \in \Y$; let $\vartheta^2(y)$ be its smallest eigenvalue.  Thus $\vartheta^2 \colon \Y \rightarrow [0,\infty)$ is a continuous $T\times \Lambda$-invariant function, equal to $1$ on $\Cyl_{\Q}$.

Define a potential
\[ V=t^2|v_{\sY}|^2+4\pi t\pair{\mu}{v}-tc_\lambda+(1+t)^2f^2\vartheta^2-(1+t)f|b_1|-(1+t)|df|\cdot|b_2|. \]
Using equation \eqref{eqn:RoughBochner}, $V$ satisfies \eqref{eqn:FredholmDeformation1}.  It is clear that $V$ is $T$-invariant, continuous.  Note that $|df|$ is bounded on $\Y \setminus \Cyl_{\Q}$, whereas on $\Cyl_{\Q}$, $|df|=|f^\prime(x)|$.  Since $b_1,b_2$ are bounded, it follows that there is a lower bound of the form
\begin{equation} 
\label{eqn:VLower}
V \ge t^2|v_{\sY}|^2+4\pi t\pair{\mu}{v}-tc_\lambda+(1+t)f^2\big((1+t)\vartheta^2-c f^{-2}(f+|f^\prime|)\big),
\end{equation}
for some constant $c$.

The growth condition \eqref{eqn:fcnf} for $f$ implies we can find $s_0>0$ such that for $s>s_0$
\begin{equation} 
\label{eqn:xlambda}
f(s)^{-2}(f(s)+|f^\prime(s)|)<\tfrac{1}{2}c^{-1}.
\end{equation}
Let $\calK \subset \Y$ be the subset where $x\le s_0$.  We consider $V$ on each of the subsets $\calK$ and $\ol{\Y \setminus \calK}$.  On $\calK$ the function $t\pair{\mu}{v}$ is proper and bounded below by Lemma \ref{lem:MuVPairing}, while the other terms in \eqref{eqn:VLower} are bounded.  Thus $V|_{\calK}$ is proper and bounded below.

Note that $\Y \setminus \calK \subset \Cyl_{\Q}$.  Since $\vartheta|_{\Cyl_{\Q}}=1$ and using \eqref{eqn:xlambda}, we obtain the simpler lower bound
\begin{equation}
\label{eqn:VLower2}
V\big|_{\Y \setminus \calK} \ge t^2|v_{\sY}|^2+4\pi t\pair{\mu}{v}-tc_\lambda+\tfrac{1}{2}(1+t)f^2,
\end{equation}
which holds on $\Y \setminus \calK$.  By Lemma \ref{lem:MuVPairing}, the function $\pair{\mu}{v}+f$ is proper and bounded below on $\Y$, and this easily implies that the right hand side of \eqref{eqn:VLower2} is proper and bounded below.
\end{proof}

\subsection{Continuity of the index.}
Let $H$ be a Hilbert space and let $a_0$, $a$ be unbounded self-adjoint operators such that $\dom(a_0)\cap \dom(a)$ is dense.  Suppose the family of operators
\[ a_t=a_0+ta, \qquad t \ge 0 \]
is essentially self-adjoint.  The \emph{bounded transform} of $a_t$ is the bounded self-adjoint operator $b(a_t)$, where $b(r)=r(1+r^2)^{-1/2}$.  It is convenient to use the following criterion adapted from Nicolaescu \cite[Proposition 1.6]{NicolaescuFredholm}.
\begin{lemma}[\cite{NicolaescuFredholm}]
\label{lem:criterion}
Let $a_t=a_0+ta$, $t \ge 0$ be a family of unbounded self-adjoint operators, as above.  Suppose that for each $t \ge 0$ the following conditions hold:
\begin{enumerate}
\item $a_t$ has a gap in its spectrum.
\item $\dom(a_t) \subset \dom(a)$
\item $a^2 \le C(a_t^2+C^\prime)$ for some $C,C^\prime >0$.
\end{enumerate}
Then the family of bounded transforms $t \mapsto b(a_t)$ is norm-continuous.
\end{lemma}
\begin{remark}
\label{rem:CommonCore}
The third condition in Lemma \ref{lem:criterion} implies $\|a \xi \| \le C^{\prime \prime}\big(\|a_t\xi\|+\|\xi\|\big)$ for some $C^{\prime \prime}>0$ and all $\xi \in \dom(a_t)$.  If the operators $a_t$ have a common core, then the estimate $(c)$ (verified on elements of the common core) implies $\dom(a_t) \subset \dom(a)$.  If the operators $a_t$ are Fredholm, then $0$ is an isolated point of the spectrum, hence in particular $a_t$ has a gap in its spectrum.  Thus in this special case the criterion amounts to proving the estimate $(c)$.  If $H$ is $\bZ_2$ graded and the $a_t$ are odd, then $b(a_t)$ is an odd self-adjoint Fredholm operator with the same index as $a_t$.  Norm-continuity of the bounded transform therefore implies $\index(a_t)$ is independent of $t$.
\end{remark}

\begin{proposition}
\label{prop:CtsTFredholm2}
For $\lambda \in \Irr(T)$ fixed, the family of bounded operators $t \mapsto b(\scr{D}_t)_\lambda$, $t \ge 0$ is norm-continuous.  Consequently $\index(\scr{D}_t)=\index(\scr{D}) \in R^{-\infty}(T)$.
\end{proposition}
\begin{proof}
Fix an isotypical component $\lambda \in \Irr(T)$.  We will prove the proposition by applying Lemma \ref{lem:criterion} and Remark \ref{rem:CommonCore} to the family of odd self-adjoint operators $a_t=(\scr{D}_t)_\lambda$ (these have a common core consisting of smooth compactly supported sections in the $\lambda$-isotypical subspace).  The operator $a=(f\theta-\i \c(v_{\sY}))_\lambda$.  Since $\c(v_{\sY})$ is bounded we can ignore this term.  Likewise since $\|\theta \| \le 1$ we can replace $\theta$ with $1$.  Thus it suffices to prove that there are constants $C,C^\prime>0$ such that
\begin{equation} 
\label{eqn:Inequality0}
f^2 \le C\big((\scr{D}_t^2)_\lambda+C^\prime \big).
\end{equation}
Using inequality \eqref{eqn:FredholmDeformation1} this amounts to showing
\begin{equation} 
\label{eqn:Inequality}
f^2 \le C(V+C^\prime).
\end{equation}
It is convenient to make use of the subset $\calK \subset \Y$ introduced in the proof of Proposition \ref{prop:CtsTFredholm}.  Since $V$ is bounded below, while $f$ is bounded on $\calK$, it is easy to ensure \eqref{eqn:Inequality} holds on $\calK$ by taking $C^\prime \gg 0$.  On $\Y \setminus \calK$, using the lower bound \eqref{eqn:VLower2}, inequality \eqref{eqn:Inequality} would follow from
\[ f^2 \le C\big(4\pi t\pair{\mu}{v}-tc_\lambda+\tfrac{1}{2}(1+t)f^2+C^\prime \big).\]
As $\pair{\mu}{v}$ is bounded below, by taking $C^\prime \gg 0$ we can ensure $4\pi t\pair{\mu}{v}+C^\prime >tc_\lambda$.  Then taking $C>2(1+t)^{-1}$ gives the result.
\end{proof}

\begin{remark}
\label{rem:GapTopology}
In fact to just show invariance of the index, it is enough to verify the weaker result that the family of resolvents $t \mapsto (\scr{D}_t\pm \i)_\lambda^{-1}$ is norm-continuous, or equivalently that $t \mapsto (\scr{D}_t)_\lambda$ is continuous in the `gap topology' on the space of closed unbounded operators (see Kato \cite[Theorem IV.2.23]{Kato}).  That this is sufficient is an older result of Cordes and Labrousse \cite{CordesLabrousse} (see Kato \cite[IV.5.17]{Kato}).  To prove norm-continuity of the resolvents, one can use
\[ (\scr{D}_t\pm \i)^{-1}-(\scr{D}_s\pm \i)^{-1}=(\scr{D}_t\pm \i)^{-1}(\scr{D}_s-\scr{D}_t)(\scr{D}_s\pm \i)^{-1}=(s-t)(\scr{D}_t\pm \i)^{-1}(f\theta-\i \c(v_{\sY}))(\scr{D}_s\pm \i)^{-1}, \]
and so it is enough to prove that $(\scr{D}_t\pm \i)_\lambda^{-1} f\theta$ is bounded.  The latter follows from inequality \eqref{eqn:Inequality0} proved in Proposition \ref{prop:CtsTFredholm2}.
\end{remark}

\begin{remark}
Here is a slightly different perspective on Proposition \ref{prop:CtsTFredholm2}, which avoids using Lemma \ref{lem:criterion} in favour of Hilbert $C^\ast$-module methods.  Let $\ul{H}=C([0,1],H)$ be the Hilbert $C[0,1]$-module consisting of continuous functions $[0,1] \rightarrow H$.  The family of self-adjoint operators $\underline{\scr{D}}=(\scr{D}_t)_{t \in [0,1]}$ defines an unbounded self-adjoint operator on $\ul{H}$ (this follows because the coefficients of $\scr{D}_t$ vary continuously, in fact smoothly).  A `localization' theorem of Pierrot \cite[Theorem 1.18]{PierrotLocalization} (see also \cite{KaadLeschLocalization}) implies $\underline{\scr{D}}$ defines a \emph{regular} self-adjoint operator on $\ul{H}$; this means functional calculus for unbounded operators on Hilbert $C^\ast$-modules is available, hence $\ul{F}=b(\underline{\scr{D}})$ is a bounded operator on $\ul{H}$.  To show that the index of $\scr{D}_t$ is independent of $t \in [0,1]$, it suffices to show that the pair $(\ul{H},\ul{F})$ defines a $\KK$-theory homotopy, i.e. a $\KK$-theory cycle for the pair of $C^\ast$-algebras $(C^\ast(T),C[0,1])$, cf. \cite{HigsonPrimer, KasparovNovikov}.  This amounts to showing that for each $\lambda \in \Irr(T)$, the resolvents $t \mapsto (\scr{D}_t \pm \i)^{-1}_\lambda$ form a norm-continuous family of compact operators (cf. \cite{HigsonPrimer}), and we explained this already in Remark \ref{rem:GapTopology}.  This is a weaker result than Proposition \ref{prop:CtsTFredholm2}, but still implies constancy of the index.
\end{remark}

\section{The Ma-Zhang-type index formula}\label{sec:MaZhangType}
Our goal in this section is to `break up' the index of $\scr{D}$ into contributions from each component of the vanishing locus $Z$.  The basic tool we use for this is the splitting theorem for elliptic boundary value problems, see Theorem \ref{thm:Split}.  Using methods of Ma and Zhang \cite{MaZhangTransEll}, we obtain a formula for $\index(\scr{D}) \in R^{-\infty}(T)$ as an infinite (but locally finite) sum of contributions labelled by the components $Z_\beta$ of the vanishing locus.  Each contribution can be described as the `limit', as the parameter $t \rightarrow \infty$, of the index of an Atiyah-Patodi-Singer (APS) \cite{AtiyahPatodiSingerI} boundary value problem $(\scr{D}_t,B_t)$ on a compact neighborhood $U_\beta$ of $Z_\beta \cap \X$, where $\X=(\phi^{\g/\t})^{-1}(0) \subset \Y$.

\subsection{The splitting theorem and $\scr{D}_t$.}
For each $\beta \in W \cdot \B$, choose a small closed ball $B_\beta \subset \t$ centered on $\beta$, such that $B_\beta \cap B_\gamma =\emptyset$ for $\beta \ne \gamma$.  Recall that the intersection $Z_\beta \cap (\Y \setminus \Cyl_{\Q})=(\Y \setminus \Cyl_{\Q})^\beta \cap \phi^{-1}(\beta)$ is compact.  Let $U_\beta \subset \phi^{-1}(B_\beta)$ be a compact $T$-invariant neighborhood of $Z_\beta \cap (\Y \setminus \Cyl_{\Q})$ in $\Y$ such that $U_\beta$ is also a manifold with smooth boundary $\partial U_\beta = N_\beta$.  By construction, for each $y \in N_\beta$ either $|v_{\sY}(y)|>0$ or $\theta^2(y)=1$.  Fix a regular value $R>0$ of the function $|\phi | \colon \Y \rightarrow [0,\infty)$, and let
\[ U_R=\bigcup_{|\beta|<R} U_\beta, \qquad N_R=\partial U_R, \qquad W_R=\ol{\Y \setminus U_R}.\]

For $t>0$ the Dirac-type operator $\scr{D}_t$ is $T$-coercive (see Proposition \ref{prop:CtsTFredholm}).  Choosing an adapted boundary operator $A_t^+$ for $\scr{D}_t^+\upharpoonright U_R$, we can apply the splitting theorem:
\begin{equation} 
\label{eqn:SplitD}
\index(\scr{D})=\index(\scr{D}_t)=\index(\scr{D}^+_t\upharpoonright U_R,B_{<0}(A_t^+))+\index(\scr{D}^+_t\upharpoonright W_R,B_{\le 0}(-A_t^+)).
\end{equation}
\begin{remark}
\label{rem:SumIndependent}
Note that in equation \eqref{eqn:SplitD} we are allowing the boundary condition to depend on the parameter $t$.  Thus the two summands on the right hand side are \emph{not} independent of $t$, although their sum is, by Proposition \ref{prop:CtsTFredholm2}.  One should compare the approach in Ma-Zhang \cite{MaZhangTransEll}, which involves a similar boundary condition depending on $t$.
\end{remark}

Equation \eqref{eqn:SplitD} is really an infinite collection of equations, one for each allowed choice of $A_t^+$.  For results below it is convenient to choose a particular adapted boundary operator.  Let $\nu$ be an inward unit normal vector for $U_R$ along $N_R$.  For the Dirac operator $\st{D}$, we will use a \emph{canonical boundary operator} $A$ given along $N_R$ by the expression
\begin{equation} 
\label{eqn:CanonicalBoundary}
A=\sigma_{\st{D}}(\nu)^{-1}\st{D}-\nabla^{\E}_{\nu}+\tfrac{\dim(N_R)}{2}h
\end{equation}
where $\sigma_{\st{D}}(\nu)^{-1}=-\c(\nu)$ and $h$ is the mean curvature of $N_R$, cf. Gilkey \cite[p.142]{GilkeyBdOp}, B{\"a}r and Ballmann \cite[Appendix A]{BarBallmannGuide}.  A useful property of this choice is that $A$ anti-commutes with $\c(\nu)$:
\[ A \c(\nu)=-\c(\nu)A. \]
Let $A^+$ (resp. $A^-$) denote the restriction of $A$ to sections of $\E^+$ (resp. $\E^-$), thus $A^-\c(\nu)=-\c(\nu)A^+$.  The operators $A^{\pm}$ are essentially self-adjoint (cf. the calculation on p. 25 of \cite{BarBallmannGuide}), hence $A^+$ is an adapted boundary operator for $\st{D}^+$.  

To obtain a `canonical' boundary operator $A_t$ for $\scr{D}_t$, take the expression for $\sigma_{\scr{D}_t}(\nu)^{-1}\scr{D}_t$ along $N_R$ and simply replace $\sigma_{\st{D}}(\nu)^{-1}\st{D}$ with $A$; thus
\[ A_t=A+\i t \c(\nu)\c(v_{\sY})-(1+t)\c(\nu)f\theta.\]
Since $N_R$ is $T$-invariant and $v_{\sY}$ lies in the tangent distribution to the orbits, $\c(\nu)$, $\c(v_{\sY})$ anti-commute.  The operators $A_t^{\pm}$ are again essentially self-adjoint, hence $A_t^+$ is an adapted boundary operator for $\scr{D}_t^+\upharpoonright U_R$, and defines an Atiyah-Patodi-Singer boundary problem $(\scr{D}_t^+,B_{<0}(A_t^+))$.  Since 
\[ -\sigma_{\scr{D}_t}(\nu) \circ A_t^+ \circ \sigma_{\scr{D}_t}(\nu)^{-1}=\c(\nu) A_t^+ \c(\nu)=A_t^-, \]
the Hilbert space adjoint is $(\scr{D}_t^-,B_{\le 0}(A^-_t))$.

\subsection{Dependence of the APS index on $t$.}
\begin{proposition}[cf. \cite{MaZhangTransEll}, Proposition 1.1]
\label{prop:InvertBoundary}
Fix $\lambda \in \Irr(T)\simeq \Lambda^\ast$.  For $t\gg 0$, $(A_t)_\lambda$ is invertible.
\end{proposition}
\begin{proof}
For the proof we temporarily suspend our convention regarding graded commutators, and write $\{\cdot,\cdot \}$ for the anti-commutator, $[\cdot,\cdot]$ for the ordinary commutator.  The calculation of $A_t^2$ is similar to the Bochner formula \eqref{eqn:Bochner00}:
\begin{align*} 
A_t^2&=A^2+t^2|v_{\sY}|^2+(1+t)^2f^2\theta^2+\i t\{A,\c(\nu)\c(v_{\sY})\}-(1+t)\{A,\c(\nu)f\theta \}\\
&=A^2+t^2|v_{\sY}|^2+(1+t)^2f^2\theta^2+\i t\c(\nu)[-A,\c(v_{\sY})]+(1+t)\c(\nu)[A,f\theta]
\end{align*}
where in the second line we used $\{A,\c(\nu)\}=0$.  In terms of a local orthonormal frame $X_1=\nu$, $X_2,...,X_{\dim(\Y)}$ for $\Y$ the operator $A$ is
\[ A=-\c(\nu)\sum_{n\ge 2} \c(X_n)\nabla^{\E}_{X_n}+\tfrac{\dim(N_R)}{2}h.\]
Using that $\c(v_{\sY})$ anti-commutes with $\c(\nu)$ we have
\[ [-A,\c(v_{\sY})]=\c(\nu)\Big\{ \sum_{n\ge 2} \c(X_n)\nabla^{\E}_{X_n},\c(v_{\sY})\Big\}.\]
Arguing as in the proof of Proposition \ref{prop:CtsTFredholm} we find that the anti-commutator in this expression is bounded on the $\lambda$-isotypical component (note that the argument is simpler than Proposition \ref{prop:CtsTFredholm}, because $N_R$ is compact).  Also, since $N_R$ is compact, the commutator $[A,f\theta]$ is a bounded bundle endomorphism.  Thus
\[ A_t^2=A^2+t^2|v_{\sY}|^2+(1+t)^2f^2\theta^2+(1+t)S \]
where $S$ is an operator which is bounded on the $\lambda$-isotypical component.  The manifold $U_R$ was chosen such that at each point of $y \in N_R$, either $|v_{\sY}(y)|>0$ or $\theta^2(y)=1$; thus $|v_{\sY}|^2+f^2\theta^2$ is a strictly positive bundle endomorphism along $N_R$.  Taking $t \gg 0$ we can ensure these terms dominate, hence $(A_t)_\lambda^2$ is invertible.
\end{proof}

\begin{corollary}
\label{cor:tDepend}
Fix $\lambda \in \Irr(T)\simeq \Lambda^\ast$.  \emph{When $t$ is sufficiently large}, the summands $\index(\scr{D}^+_t\upharpoonright U_R,B_{<0}(A_t^+))_\lambda$ and $\index(\scr{D}^+_t\upharpoonright W_R,B_{\le 0}(-A_t^+))_\lambda$ on the right hand side of \eqref{eqn:SplitD} are separately independent of $t$.
\end{corollary}
\begin{proof}
By Remark \ref{rem:SumIndependent} it is enough to prove this for $\index(\scr{D}^+_t\upharpoonright U_R,B_{<0}(A_t^+))_\lambda$.  For ease of reading, for the remainder of the proof we will omit `$\upharpoonright U_R$' from the notation.  By Proposition \ref{prop:InvertBoundary}, there is some $t_\lambda$ such that for all $t \ge t_\lambda$, the adapted boundary operator $(A_t)_\lambda$ is invertible.  Thus the Hilbert space adjoint of $(\scr{D}_t^+,B_{<0}(A_t^+))_\lambda$ is $(\scr{D}_t^-,B_{<0}(A_t^-))_\lambda$ (i.e. we may omit the $0$-eigenspace), and so the index of $(\scr{D}_t^+,B_{<0}(A_t^+))_\lambda$ is the index in the $\bZ_2$-graded sense of the odd \emph{self-adjoint} operator $(\scr{D}_t,B_{<0}(A_t))_\lambda$.  

We will prove that the index of the 2-parameter family $(s,t) \mapsto (\scr{D}_s,B_{<0}(A_t))_\lambda$, $s \ge 0$, $t\ge t_\lambda$ is independent of $(s,t)$.  First fix $t$ and consider the dependence on $s$.  By Lemma \ref{lem:criterion} and Remark \ref{rem:CommonCore}, norm-continuity of the bounded transforms follows from an estimate of the form appearing in Lemma \ref{lem:criterion}.  But in this case the operator $a=f\theta-\i \c(v_{\sY})$ is bounded on the compact space $U_R$, so the estimate holds. 

Next fix $s$ and consider the dependence on $t \ge t_\lambda$.  The idea is that by continuity of the spectrum, invertibility of $(A_t)_\lambda$ for $t \ge t_\lambda$ implies that the boundary condition $B_{<0}(A_t)_\lambda$ varies `continuously' with $t$, since no eigenvalues can cross $0$.  This in turn implies constancy of the index.  A discussion of continuous families of boundary conditions can be found, for example, in B{\"a}r and Ballmann \cite[Section 8.2]{BarBallmann}.  Let $P_{<0}(A_t)$ denote the $L^2$-orthogonal projection onto $B_{<0}(A_t)$, and let $r \ge 0$.  Since $\partial U_R$ is compact, the dense subspace $\dom(|A_t|^r) \subset L^2(\partial U_R,\E)$ does not depend on $t$ and defines the level $r$ Sobolev space $H^{(r)}=L^2_r(\partial U_R, \E)$.  It follows from the spectral theorem that $P_{<0}(A_t)$ induces a bounded linear operator $P_{<0}(A_t)^{(r)}$ in $H^{(r)}$.  According to a result in \emph{loc. cit.}, it suffices to prove that the family $t \mapsto P_{<0}(A_t)^{(r)}_\lambda$, $t \ge t_\lambda$ is norm-continuous with respect to the operator norm on $\bB(H^{(r)}_\lambda)$ for $r=1/2$.  A short, self-contained proof of this fact (for arbitrary $r$) can be found in \cite[Theorem 3.2]{KirkKlassen}.
\end{proof}

By Corollary \ref{cor:tDepend}, it makes sense to define 
\[ \index_{\tn{APS},\beta}(\scr{D},v)=\lim_{t \rightarrow \infty} \index(\scr{D}^+_t\upharpoonright U_\beta,B_{<0}(A_t^+)) \]
as well as
\begin{equation} 
\label{eqn:APSsum}
\index_{\tn{APS},R}(\scr{D},v)=\lim_{t\rightarrow \infty}\index(\scr{D}^+_t\upharpoonright U_R,B_{<0}(A_t^+))=\sum_{|\beta|<R} \index_{\tn{APS},\beta}(\scr{D},v),
\end{equation}
with the convergence being in $R^{-\infty}(T)$.  Taking the limit $t \rightarrow \infty$ of \eqref{eqn:SplitD} we find
\begin{equation}
\label{eqn:Limt}
\index(\scr{D})=\index_{\tn{APS},R}(\scr{D},v)+\lim_{t \rightarrow \infty}\index(\scr{D}^+_t\upharpoonright W_R,B_{\le 0}(-A_t^+)).
\end{equation}

\subsection{Dependence of the APS index on $R$.}
\begin{theorem}[compare \cite{MaZhangTransEll} Theorem 2.1]
\label{prop:RDepend}
The limit
\[ \lim_{R \rightarrow \infty} \lim_{t\rightarrow \infty} \index(\scr{D}^+_t\upharpoonright W_R,B_{\le 0}(-A_t^+))=0 \]
in $R^{-\infty}(T)$.
\end{theorem}
In part of the proof we will use a method we learned from \cite[pp. 27--29]{MaZhangTransEll}.
\begin{proof}
Fix $\lambda \in \Lambda^\ast \simeq \Irr(T)$.\ignore{We must show that there exists a constant $R_\lambda$ such that for each $R>R_\lambda$, 
\[ \lim_{t \rightarrow \infty} \index(\scr{D}^+_t\upharpoonright W_R,B_{\le 0}(-A_t^+))_\lambda=0.\]}\ignore{
the following holds: there exists a constant $t_{\lambda,R}$ such that for all $t>t_{\lambda,R}$ the index $\index(\scr{D}^+_t\upharpoonright W_R,B_{\le 0}(-A_t^+))_\lambda=0$.
}
We will prove that when $R$ is sufficiently large there is a constant $t_{\lambda,R}$ such that for $t>t_{\lambda,R}$ the inequality
\[ \|\scr{D}_t s\|^2 \ge \tfrac{1}{2}\|\nabla^{\E} s\|^2+(t-t_{\lambda,R})\|s\|^2 \]
holds for all $s \in C^\infty_c(W_R,\E)_\lambda$ satisfying $(s,A_ts)_{N_R}\ge 0$.  Hence when $t>t_{\lambda,R}$ both the kernel and cokernel of $(\scr{D}_t^+\upharpoonright W_R,B_{\le 0}(-A_t^+))_\lambda$ vanish separately (recall that the adjoint operator is $(\scr{D}^-_t\upharpoonright W_R,B_{<0}(-A_t^-))$), which implies the result.

By Green's formula,
\begin{equation} 
\label{eqn:GreenV}
\|\scr{D}_ts\|^2=(s,\scr{D}_t^2s)_{W_R}-(s,\c(\nu)\scr{D}_ts)_{N_R} 
\end{equation}
where we have used that $-\nu$ is the inward unit normal vector for $W_R$, and the skew-adjointness of $\c(\nu)$.\ignore{It is necessary to use Green's formula here, even if $\scr{D}_t$ with boundary condition is self-adjoint, because the other argument $\scr{D}_t s$ of the inner product might not satisfy the boundary condition ($s$ on its own does, but $\scr{D}_t$ need not preserve this fact), i.e. might not be in the domain of the Hilbert space operator.  The usual relation for symmetric/self-adjoint operators is only valid when the vectors are in the domain of the operator in question.}

For the first term in \eqref{eqn:GreenV} we use the lower bound $(\scr{D}_t^2)_\lambda \ge (\st{D}^2+V)_\lambda$ proved in Proposition \ref{prop:CtsTFredholm}, and the lower bound for the potential $V$ in equation \eqref{eqn:VLower}:
\[V \ge t^2|v_{\sY}|^2+4\pi t\pair{\mu}{v}-tc_\lambda+(1+t)f^2\big((1+t)\vartheta-c f^{-2}(f+|f^\prime|)\big). \] 
By \eqref{eqn:fcnf}, $cf^{-2}(f+|f^\prime|)$ is bounded globally on $\Y$ by some constant $c^\prime$. 
\ignore{ functions $f^{-1}$, $|f^\prime|f^{-2}$ are bounded globally on $\Y$, hence there exists a constant $c_4(\lambda)$ such that
\[ c_4(\lambda)\ge c_1(\lambda)f^{-2}+c_2f^{-1}+c_3|f^\prime|f^{-2}. \]}Assume $t>1$ so $2t>t+1$, thus
\begin{equation} 
\label{eqn:BoundtV}
t^{-1}V \ge t|v_{\sY}|^2+4\pi\pair{\mu}{v}-c_\lambda+f^2(t\vartheta^2-2c^\prime).
\end{equation}
We claim that for $R$ and $t$ sufficiently large we have $V|_{W_R} \ge t$; indeed, we may verify this separately on $\Cyl_{\Q}$, $W_R \setminus \Cyl_{\Q}$:
\begin{enumerate}[leftmargin=*]
\item On $\Cyl_{\Q}$, $\vartheta \equiv 1$ and $f \ge 1$, hence  
\[ t^{-1}V\ge 4\pi \pair{\mu}{v}-c_\lambda+f^2(t-2c^\prime).\]
Since $\pair{\mu}{v}$ is bounded below, for $t \gg 0$ we will have $t^{-1}V \ge 1$.
\item On $W_R \setminus \Cyl_{\Q}$, $f=1$.  Dropping the non-negative term $\vartheta^2$, we have
\[ t^{-1}V \ge t|v_{\sY}|^2+4\pi\pair{\mu}{v}-c_\lambda-2c^\prime. \]
On $\Y \setminus \Cyl_{\Q}$ the function $\pair{\mu}{v}$ is proper and bounded below.  Hence the subset $\calK_\lambda \subset \Y \setminus \Cyl_{\Q}$ where $4\pi \pair{\mu}{v}\le c_\lambda+2c^\prime+1$ is compact.  By taking $R$ sufficiently large---\emph{thus excising sufficiently many components of the vanishing locus $Z \cap (\Y \setminus \Cyl_{\Q})$ from $\Y \setminus \Cyl_{\Q}$}---say $R>R_\lambda$, we can arrange that $v_{\sY}$ does not vanish on $\calK_\lambda \cap W_R$.  By compactness $|v_{\sY}|$ is bounded below by some positive constant on $\calK_\lambda \cap W_R$, hence taking $t \gg 0$ (depending on $R$) we can ensure
\[ t|v_{\sY}|^2+4\pi \pair{\mu}{v}-c_\lambda-2c^\prime \ge 1. \]
\end{enumerate}
\ignore{In either case $t^{-1}V \ge 1$, hence by \eqref{eqn:BoundtV}:
\[ t^{-1}V|_{W_R} \ge 1 \qquad \Rightarrow \qquad V|_{W_R} \ge t.\]}
We have thus shown that for $R>R_\lambda$ and $t$ sufficiently large (depending on $R$)
\begin{equation}
\label{eqn:BoundTerm1}
(s,\scr{D}_t^2 s)_{W_R} \ge (s,\st{D}^2s)_{W_R}+(s,Vs)_{W_R} \ge (s,\st{D}^2s)_{W_R}+t\|s\|^2.
\end{equation}

Now consider the second term in \eqref{eqn:GreenV}.  For the remainder of the proof we write $\nabla$ in place of $\nabla^{\E}$ to make expressions a little cleaner.  Along $N_R$ we have
\[ -\c(\nu)\scr{D}_t s=A_ts+\nabla_\nu s-\tfrac{\dim(N_R)}{2}hs,\]
where $h$ is the mean curvature of $N_R$.  By assumption $(s,A_ts)_{N_R} \ge 0$ hence, dropping this term,
\begin{equation} 
\label{eqn:SecondTerm}
-\big(s,\c(\nu)\scr{D}_ts\big)_{N_R} \ge \big(s,\nabla_\nu s\big)_{N_R}-\tfrac{\dim(N_R)}{2}(s,hs)_{N_R}.
\end{equation}
To obtain an expression for $(s,\nabla_\nu s)$, apply Green's formula to the operator $\nabla=\nabla^{\E}$ on $W_R$:
\begin{align*}
\|\nabla s\|^2&=(s,\nabla^\ast \nabla s)_{W_R}-(\sigma_\nabla(-\nu)s,\nabla s)_{N_R}\\
&=(s,\nabla^\ast \nabla s)_{W_R}+(\nu \otimes s,\nabla s)_{N_R}\\
&=(s,\nabla^\ast \nabla s)_{W_R}+(s,\iota(\nu)\nabla s)_{N_R}
\end{align*} 
where in the last line $\iota(\nu)$ denotes the contraction operator defined using the metric.  Thus
\[ (s,\nabla_\nu s)_{N_R}=\|\nabla s\|^2-(s,\nabla^\ast \nabla s)_{W_R}.\]
By the Lichnerowicz-Weitzenbock formula,
\[ \st{D}^2=\nabla^\ast \nabla+\R \]
where $\R$ is a bundle endomorphism of $\E$, depending on the metrics of $E$, $S$, $T\Y$ and the choices of connections.  Because these structures take a product form on $\Cyl_{\Q}$ and are $\hLambda$-invariant, $\R$ is bounded globally on $\Y$ by some constant $r$.  Substituting these expressions in \eqref{eqn:SecondTerm} we have
\begin{equation}
\label{eqn:BoundTerm2}
-\big(s,\c(\nu)\scr{D}_ts\big)_{N_R} \ge \|\nabla s\|^2-(s,\st{D}^2 s)_{W_R}-r\|s\|^2-\tfrac{\dim(N_R)}{2}(s,hs)_{N_R}.
\end{equation} 

Substituting \eqref{eqn:BoundTerm1}, \eqref{eqn:BoundTerm2} into \eqref{eqn:GreenV} yields, for $R>R_\lambda$ and all $t$ sufficiently large (depending on $R$):
\[ \|\scr{D}_ts\|^2 \ge \|\nabla s\|^2-\tfrac{\dim(N_R)}{2}(s,hs)_{N_R}+(t-r)\|s\|^2.\]
As $N_R$ is compact, $h$ is bounded above by a constant depending only on $R$.  For any $\delta \in (0,1)$ there is an estimate (cf. \cite[Theorem 1.5.1.10]{Grisvard})
\[ \|s\|^2_{N_R} \le (1+\delta^{-1})\|s\|_{W_R}^2+\delta \|\nabla s\|_{W_R}^2 \]
Choosing $\delta$ to be sufficiently small we obtain an estimate of the form
\[ \|\scr{D}_ts\|^2 \ge \tfrac{1}{2}\|\nabla s\|^2+(t-t_{\lambda,R})\|s\|^2\]
for some constant $t_{\lambda,R}$.
\end{proof}

\subsection{The `limit of APS index' formula.}
By Theorem \ref{prop:RDepend} we can take the limit as $R \rightarrow \infty$ of \eqref{eqn:Limt}:
\begin{equation} 
\label{eqn:LimitOfAPS}
\index(\scr{D})=\lim_{R \rightarrow \infty}\index_{\tn{APS},R}(\scr{D},v)=\sum_{\beta}\index_{\tn{APS},\beta}(\scr{D},v).
\end{equation}
For the second equality we have used \eqref{eqn:APSsum}.  This is the Ma-Zhang-type `limit of APS index' formula for $\index(\scr{D})$; it expresses the index as a sum of contributions $\index_{\tn{APS},\beta}(\scr{D},v)$ labelled by the components of the vanishing locus $Z$.  $Z$ has infinitely many components, but as a consequence of Theorem \ref{prop:RDepend}, the sum is locally finite, in the sense that for any fixed $\lambda \in \Irr(T)$, $\index_{\tn{APS},\beta}(\scr{D},v)_\lambda=0$ for all but finitely many $\lambda$.

\begin{remark}
For a non-compact prequantized Hamiltonian $G$-space with proper moment map, the analogue of $\index(\scr{D})$ is not defined in general.  In their proof of the Vergne conjecture, Ma-Zhang \cite{MaZhangVergneConj,MaZhangTransEll} showed that $\index_{\tn{APS}}(\scr{D},v)$, defined using a deformation and limits $R,t\rightarrow \infty$ similar to above, is well-defined.  The resulting `quantization' of $M$ satisfies the $[Q,R]=0$ Theorem and behaves functorially under restriction to subgroups.
\end{remark}

\ignore{
By Theorem \ref{prop:RDepend} and Corollary \ref{cor:tDepend} we can choose $R_\lambda$ such that for $R>R_\lambda$
\[ \index(\scr{D}^+_t\upharpoonright W_R,B_{\le 0}(-A_t^+))_\lambda=0 \]
for all $t$ larger than some constant $t_{\lambda,R}$ depending on $R$.  Thus by \eqref{eqn:SplitD} and Proposition \ref{prop:CtsTFredholm},
\[ \index(\scr{D})_\lambda=\index_{\tn{APS},R}(\scr{D},v)_\lambda \]
for all $R>R_\lambda$.  In particular the right hand side stabilizes as $R \rightarrow \infty$ and we can define $\index_{\tn{APS}}(\scr{D},v) \in R^{-\infty}(T)$ with multiplicities
\[ \index_{\tn{APS}}(\scr{D},v)_\lambda=\lim_{R\rightarrow \infty} \index_{\tn{APS},R}(\scr{D},v)_\lambda, \]
satisfying $\index(\scr{D})=\index_{\tn{APS}}(\scr{D},v)$.  As a consequence of equation \eqref{eqn:APSsum}:
\begin{equation} 
\label{eqn:LimitOfAPS}
\index(\scr{D})=\index_{\tn{APS}}(\scr{D},v)=\sum_{\beta}\index_{\tn{APS},\beta}(\scr{D},v).
\end{equation}
}

\section{The Paradan-type index formula}\label{sec:Paradan}
In this section we explain how to express the `limit of APS index' contributions $\index_{\tn{APS},\beta}(\scr{D},v)$ as indices of transversally elliptic symbols, resulting in a formula similar to that of Paradan \cite{ParadanRiemannRoch} (see also \cite{WittenNonAbelian}) in the compact case.  We will be somewhat brief as our strategy follows along similar lines to Ma-Zhang \cite[Section 1.4]{MaZhangTransEll}, \cite{MaZhangBravermanIndex}.  Although the situation is similar to \cite{MaZhangTransEll}, we cannot quite apply their results immediately: for example, the operator we consider has an additional zeroth order term (containing $\theta$), requiring small modifications.  Throughout this section $\beta$ will be fixed, and we set $U=U_\beta$, $N=N_\beta$ to simplify the notation.

\subsection{A Braverman-type operator.}
Following the strategy of Ma-Zhang, the first step is to study a family of operators $\scr{D}^{M}_t$, $t>0$ on an open manifold $M \supset U$ of Braverman-type (cf. \cite{Braverman2002}), which extend $\scr{D}_t$ on $U$ and such that $\index_{\tn{APS},\beta}(\scr{D},v)=\index(\scr{D}^M_t)$ for all $t>0$.  The point is that we end up with the ordinary index (rather than a limit of indices) of a Dirac-type operator on $M$, and this is a little closer to the usual setting for transversally elliptic symbols.

Recall that $U=U_\beta \subset \Y$ is a compact manifold with boundary $\partial U=N$.  Let $M$ be a relatively compact, collaring open neighborhood of $U$ in $\Y$, such that $\ol{M} \cap Z_\gamma=\emptyset$ for $\gamma \ne \beta$.  There is a $T$-equivariant diffeomorphism
\[ M \simeq U \bigcup_N N \times [0,\infty)\]
such that the outward normal vector $-\nu=\partial_r$, $r \in [0,\infty)$.  By pullback from $\Y$, we may consider $E$ and $S$ as $T$-equivariant vector bundles over $M$, and $v$, $f$, $\theta$ as smooth sections of the appropriate bundles over $M$.  

Define a new metric $g_M$ on $M$ by patching together the given metric on $U \cup_N N \times [0,1)$ (viewed as an open subset of $\Y$) with a cylinder metric of the form $dr^2+g_N$ on $N \times (0,\infty)$, using a partition of unity.  Thus $M$ becomes a complete manifold with cylindrical end $N \times (1,\infty)$, where the metric takes the form
\[ g_M|_{N\times (1,\infty)}=dr^2+g_N.\]  
Similarly we define Hermitian metrics and compatible connections on $E$, $S$ by patching together the given metrics and connections on $U \cup_N N \times [0,1)$ with metrics and connections on $N \times (0,\infty)$ that are independent of $r$.  In this way we obtain a new (essentially self-adjoint) Dirac operator $\st{D}^M$ acting on sections of $\E$ over $M$, which extends the original $\st{D}$ on $U$; note that this operator differs from the Dirac operator on $\Y$, because we have modified the metric and connections on the collar neighborhood $N \times [0,\infty)$.

Choose a smooth monotone function $h \colon [0,\infty) \rightarrow [1,\infty)$ such that $h(r)=e^r$ for $r\ge 1$ and $h(r)=1$ for $r \in [0,1/2)$.  View $h$ as a function on $N \times [0,\infty)$ and extend it identically by $1$ to $M$.  Define
\[ \scr{D}_t^M=\st{D}^M_{\theta,t}-\i th\c(v_M), \qquad \st{D}^M_{\theta,t}=\st{D}^M+(1+t)hf\theta.\]
On $U$ this agrees with the operator $\scr{D}_t$.  
\begin{remark}
Essentially we have attached a cylindrical end to $U$, and multiplied the zeroth order terms in $\scr{D}_t$ by the function $h$ which blows up at infinity in $M$.  It is not particularly important at this stage that $h$ blows up exponentially, but this will be convenient in the next subsection.
\end{remark}
\begin{proposition}
\label{prop:DMFredholm}
For $t \ge 1$, the operator $\scr{D}_t^M$ is $T$-Fredholm.  For each $\lambda \in \Irr(T)$, the family of bounded transforms $t \mapsto b(\scr{D}_t^M)_\lambda$ is norm continuous. 
\end{proposition}
\begin{remark}
The result holds for any $t>0$, although we will not need this.
\end{remark}
\begin{proof}
We go through the argument somewhat rapidly, as it is similar to (but easier than) the proofs of Propositions \ref{prop:CtsTFredholm}, \ref{prop:CtsTFredholm2}.  Using a Bochner formula for $(\scr{D}_t^M)^2$ one finds
\ignore{The Bochner formula for $\scr{D}_t^M$ reads
\begin{equation} 
\label{eqn:BochnerM1}
(\scr{D}_t^M)^2=(\st{D}^M_{\theta,t})^2+t^2h^2|v_M|^2+4\pi th\pair{\mu+\mu_E}{v}+2\i t hv_j \L^{\E}_{\xi^j}+\i thb-\i t\c(dh)\c(v_M) 
\end{equation}
where
\[(\st{D}^M_{\theta,t})^2=(\st{D}^M)^2+(1+t)^2h^2f^2\theta^2+(1+t)h\c(df)\theta+(1+t)hf\c(\nabla^{\End(E)}\theta)+(1+t)\c(dh)f\theta.\]
Thus we see that \eqref{eqn:BochnerM1} takes the form
}
\begin{equation}
\label{eqn:BoundedM1}
(\scr{D}_t^M)^2=(\st{D}^M)^2+t^2h^2|v_M|^2+(1+t)^2h^2f^2\theta^2+(1+t)hb_1+(1+t)\c(dh)b_2+2\i t hv_j \L^{\E}_{\xi^j}
\end{equation}
where $b_1$, $b_2$ are bundle endomorphisms.  One verifies that $b_1$, $b_2$ are bounded (uniformly in $t$) using a combination of the facts that: (1) both $f$, $df$ are bounded on $M$, (2) on the cylindrical end $\xi^j_M$ is tangent to $N$ and is independent of $r$, (3) with respect to the metric $g_{\Y}$ on $\Y$, the length of the vector $\partial_r$ goes to zero at $\partial \ol{M}$; consequently in an orthonormal frame (for $M$) adapted to the cylindrical end, the $\partial_r$-component $\partial_r v_j$ of the gradient of $v_j$ goes to zero at $\partial \ol{M}$.  

On the $\lambda$-isotypical component the Lie derivative $2v^j\L^{\E}_{\xi^j}$ is bounded by a constant $c_\lambda$.  Also $\theta^2 \ge \vartheta^2$, where $\vartheta^2(m)$ is the smallest eigenvalue of $\theta^2(m)$.  Define
\[ V=t^2h^2|v_M|^2+(1+t)^2h^2f^2\vartheta^2-tc_\lambda h-(1+t)h|b_1|-(1+t)|dh|\cdot |b_2|,\]
so $(\scr{D}_t^M)^2_\lambda \ge (\st{D}^2+V)_\lambda$.  A small rearrangement (using also $t \ge 1$) shows
\begin{equation}
\label{eqn:VBoundM1}
V \ge th^2\big(t|v_M|^2+tf^2\vartheta^2-c^\prime_\lambda h^{-2}(h+|h^\prime|)\big)
\end{equation}
for some constant $c^\prime_\lambda$.  

By assumption $|v_M|^2+f^2\vartheta^2>0$ on $N \times [0,\infty) \subset M$, and since the latter corresponds to a relatively compact subset of $\Y$, $|v_M|^2+f^2\vartheta^2$ is bounded below by some constant $\epsilon>0$ on $N \times [0,\infty)$.  There is an $s_0>0$ such that $s>s_0$ implies $e^{-2s}(2e^s)<\tfrac{\epsilon}{2}(c^\prime_\lambda)^{-1}$.  Let $\calK=U\cup_N N \times [0,s_0]$, a compact subset of $M$.  Then $V|_{\calK}$ is proper and bounded below.  On $M \setminus \calK=N \times (s_0,\infty)$, we have $h=e^r$ where $r \in (s_0,\infty)$, and we obtain the simpler bound
\begin{equation}
\label{eqn:VBoundM2}
V|_{M \setminus \calK} \ge \epsilon t(t-\tfrac{1}{2}) e^{2r}.
\end{equation}
Since we assumed $t \ge 1$, this shows $V$ is also proper and bounded below on the closure of $M \setminus \calK$ in $M$, and completes the proof that $\scr{D}_t^M$ is $T$-Fredholm.

For the norm-continuity, as in Proposition \ref{prop:CtsTFredholm2} it suffices to prove an inequality of the form
\[ (h^2|v_M|^2+h^2f^2\theta^2)_\lambda \le C((\scr{D}_t^M)^2+C^\prime)_\lambda\]
for some $C,C^\prime>0$ (which may depend on $t$).  The function $|v_M|^2+f^2|\theta|^2$ is bounded on $M$ by some constant $c$.  Thus it suffices to find $C,C^\prime$ such that
\begin{equation} 
\label{eqn:ReqdBoundM}
c h^2 \le C(V+C^\prime).
\end{equation}
This is easy to ensure on the compact set $\calK$ where $h$ is bounded by taking $C,C^\prime \gg 0$.  On the other hand on $M \setminus \calK$, $h=e^r$ and we may use \eqref{eqn:VBoundM2}, thus \eqref{eqn:ReqdBoundM} is implied by
\[ c e^{2r} \le C\big(\epsilon t(t-\tfrac{1}{2})e^{2r}+C^\prime \big)\]
which holds when $C^\prime \ge 0$ and $C>2c \epsilon^{-1}$.
\end{proof}

Applying the splitting theorem to the partition $M=U \cup_N \big(N\times [0,\infty)\big)$ and using that the restriction of $\scr{D}^M_t$ to $U$ agrees with $\scr{D}_t$, we have
\[ \index(\scr{D}_t^M)=\index(\scr{D}^+_t\upharpoonright U,B_{<0}(A_t^+))+\index(\scr{D}_t^{M,+}\upharpoonright N\times [0,\infty),B_{\le 0}(-A_t^+)).\]
Proposition \ref{prop:InvertBoundary} showed that for each $\lambda \in \Irr(T)$, $(A_t)_\lambda$ is invertible for $t \gg 0$.  Thus taking a limit as $t \rightarrow \infty$, and using the fact that the left hand side is independent of $t>0$, we obtain
\[ \index(\scr{D}_t^M)=\index_{\tn{APS},\beta}(\scr{D},v)+\lim_{t\rightarrow \infty}\index(\scr{D}_t^{M,+}\upharpoonright N\times [0,\infty),B_{\le 0}(-A_t^+)).\]

\begin{proposition} $\lim_{t\rightarrow \infty}\index(\scr{D}_t^{M,+}\upharpoonright N\times [0,\infty),B_{\le 0}(-A_t^+))=0$.
\end{proposition}
\begin{proof}
The proof is analogous (but much easier) than Theorem \ref{prop:RDepend}; as in that case, we prove that there is a constant $t_{\lambda}$ such that for $t>t_{\lambda}$ the inequality
\[ \|\scr{D}^{N\times [0,\infty)}_t s\|^2 \ge \tfrac{1}{2}\|\nabla s\|^2+(t-t_{\lambda})\|s\|^2 \]
holds for all smooth compactly supported sections $s \in C^\infty_c(N\times [0,\infty),\E)_\lambda$ satisfying $(s,A_ts)_{N}\ge 0$.  Using Green's formula, this involves finding lower bounds for a term from the interior of $N\times [0,\infty)$ and a term from the boundary.  The term from the interior $N\times (0,\infty)$ is handled using \eqref{eqn:VBoundM1}, noting that for $t$ sufficiently large, the summands which are quadratic in $t$ dominate, since $|v_M|^2+\theta^2>0$ on $N\times [0,\infty)$.  The term from the boundary is handled as in Theorem \ref{prop:RDepend}.
\end{proof}

\begin{corollary}
For all $t\ge 1$ $\index(\scr{D}_t^M)=\index_{\tn{APS},\beta}(\scr{D},v)$.
\end{corollary}

\subsection{Deformation to a transversally elliptic operator.}\label{sec:DefTransEll}
The next step is to describe a transversally elliptic operator on a compact manifold containing $M$ with the same index as $\scr{D}_t^M$.  We follow a strategy of Braverman \cite[Section 14]{Braverman2002}.  The analytic details of this strategy were elaborated by Ma-Zhang \cite{MaZhangBravermanIndex}.

Recall $M=U \cup_N N \times [0,\infty)$, and $r$ denoted the second projection $N \times [0,\infty) \rightarrow [0,\infty)$.  The metrics and connections take a product form on the cylindrical end $\Cyl_N=N \times (1,\infty)$.  Introduce the new coordinate $w=r^{-1}$ on $N \times (\tfrac{1}{2},\infty)$.  In terms of $w$ the cylindrical end of $M$ is $\Cyl_N=N \times (0,1)$, with $w \rightarrow 0$ being at infinity in $M$, and
\begin{equation} 
\label{eqn:Mmetricw}
g_M|_{\Cyl_N}=dr^2+g_N=w^{-4}dw^2+g_N,
\end{equation}
where $g_N$ is a Riemannian metric on $N$.

Let $DM$ denote the double of $M$, a compact $T$-manifold constructed by gluing a second copy $M^\prime$ of $M$ with reversed orientation along the cylindrical ends:
\[ DM=M \cup_{\Cyl_N} M^\prime.\]
The gluing identifies $N \times \{w\} \subset M$ with $N \times \{1-w \} \subset M^\prime$.  The construction of $DM$ has a $\bZ_2$ symmetry about the hypersurface $w=\tfrac{1}{2}$; in particular it follows that the coordinate function $w$ extends beyond $M$, and identifies a neighborhood of $\Cyl_N$ in $DM$ with $N \times (-1,2)$.  Choose a $T$-invariant Riemannian metric $g_{DM}$ on $DM$ that has a product form on $\Cyl_N$:
\begin{equation} 
\label{eqn:DMmetricw}
g_{\sDM}|_{\Cyl_N}=dw^2+g_N=r^{-4}dr^2+g_N.
\end{equation}
Comparing \eqref{eqn:Mmetricw}, \eqref{eqn:DMmetricw}, the Riemannian volume elements are related by a Jacobian factor $w^{-2}=r^2$ on $\Cyl_N$.

Considering the zeroth order term in $\scr{D}_1^M$, we are led to consider the odd, self-adjoint bundle map:
\begin{equation}
\label{eqn:defpsi}
\psi=2f\theta\wh{\otimes}1-\i\wh{\otimes} \c(v_M).
\end{equation}
In terms of $\psi$, $\scr{D}_1^M=\st{D}^M+h\psi$.  We will write $\psi^+$ (resp. $\psi^-=(\psi^+)^\ast$) when we wish to emphasize the component of $\psi$ mapping $\E^+$ to $\E^-$ (resp. $\E^-$ to $\E^+$).  

Choose a vector bundle $F$ on $M$ such that $\V^+:=\E^+\oplus F$ is trivial, and let $\V^-=\E^-\oplus F$, $\V=\V^+\oplus \V^-$.  The bundle endomorphism $\psi$ of $\E$ extends to the bundle endomorphism $\psi\oplus \id_F$ of $\V$.  In order to simplify notation, when there is little risk of confusion we will denote this extended bundle endomorphism by $\psi$.

On $\Cyl_N$, $\psi$ is invertible, hence $\psi^+ \colon \V^+\rightarrow \V^-$ restricts to an isomorphism over $\Cyl_N$.  In particular the restriction of $\V^-$ to $\Cyl_N$ is trivial as well.  Using the trivializations on $\Cyl_N$, we extend $\V$ and $\psi$ trivially from $M$ to $DM$.  To keep the notation simple, we will continue to denote these extensions by $\V$, $\psi$ respectively.  Thus $\V^{\pm}\upharpoonright M^\prime$ have fixed trivializations, in terms of which $\psi^{\pm}\upharpoonright M^\prime$ become the identity map.

The operator $\st{D}^M$ on $M$ acting on sections of $\E^{\pm}$ extends to the operator $\st{D}^M\oplus 0_F$ acting on sections of $\V^{\pm}=\E^{\pm}\oplus F$, where $0_F \colon F \rightarrow F$ denotes the zero operator.  In order to simplify notation we will denote this extended operator by $\st{D}^M$.  Consider the function $\rho$ on $M$ given by
\[ \rho:=h^{-1}, \qquad \rho|_{\Cyl_N}=h^{-1}=e^{-r}=e^{-1/w}. \] 
It follows from the formula for $\rho|_{\Cyl_N}$ that $\rho$ can be extended \emph{smoothly} by $0$ to $DM$.  We continue to denote this extension by $\rho$.

Similar to Braverman \cite[Section 14]{Braverman2002}, we define a differential operator on $DM$ that we will denote
\[ \rho \st{D}^{M,+}, \]
as follows.  Given $s \in C^\infty(DM,\V^+)$, we first restrict to $M \subset DM$ to obtain a section $s|_M \in C^\infty(M,\V^+)$, to which we can apply the operator $\st{D}^{M,+}$ followed by multiplication by $\rho|_M$.  The result is a smooth section of $\V^-|_M$, which we then extend by $0$ to a section of $\V^-$.

\begin{lemma}
$\rho \st{D}^{M,+}$ defined above is a first order differential operator on $DM$ with support contained in the closure of $M$ in $DM$.
\end{lemma}
\begin{proof}
To prove the lemma it suffices to show that the differential operator $\st{D}^{M,+} \upharpoonright \Cyl_N$ (recall $\Cyl_N=N \times (0,1)$ in terms of the coordinate $w$) can be smoothly extended to a differential operator over $N \times (-1,1) \subset DM$.  Since $\rho$ vanishes identically on $DM \setminus M$, the result follows.

It is convenient to use the `canonical' adapted operator $A^N$ for $\st{D}^M$ along the hypersurface $N \times \{\tfrac{1}{2}\}$, see equation \eqref{eqn:CanonicalBoundary}.  Since $N \times \{\tfrac{1}{2}\}$ is contained in the cylindrical end, on $\Cyl_N$ we have an equality
\[ \st{D}^M=\c(\nu)\nabla^{\E}_{\nu}+\c(\nu)A^N \]
where $\nu=-\partial_r$ is the inward unit normal vector for the metric $g$ on $M$ used to define $\st{D}^M$.  Since $A^N$ is a differential operator on $N$, we may regard it as a differential operator on $N \times (-1,1)$ which is independent of $w \in (-1,1)$.  Recall that the metrics and connections take a product form on $\Cyl_N$, so extend to $N \times (-1,1)$.  Since $-\partial_r=w^2\partial_w$ extends smoothly past $w=0$, there is no difficulty in extending the operator $\nabla^{\E}_{\nu}$ to $N \times (-1,1)$.  The endomorphism $\c(\nu)$ satisfies $\c(\nu)^2=-1$, and is independent of $w$ on $\Cyl_N$, so also extends.
\end{proof}

As $\rho \st{D}^{M,+}$ is a first-order differential operator, we may consider it as an unbounded Hilbert space operator with domain $H^1(DM,\V^+)\subset L^2(DM,\V^+)$.  Fix a $T$-equivariant invertible pseudo-differential operator $R$ acting on sections of $\V^+$ over $DM$ with symbol 
\[ \sigma_R(\xi)=\langle \xi \rangle^{-1}:=(1+|\xi|_{\sDM}^2)^{-1/2}.\]  
Similar to Braverman \cite[Section 14]{Braverman2002} we define a continuous family $t \in [0,1]$ of zeroth order pseudodifferential operators on $DM$
\begin{equation} 
\label{eqn:DefPt}
P_t=(1-t)\psi^+ +t\psi^+ R + \rho \st{D}^{M,+} R.
\end{equation}
Then $P_t$ extends to a bounded linear operator $L^2(DM,\V^+) \rightarrow L^2(DM,\V^-)$.

Recall that a pseudo-differential operator $P$ on a compact $G$-manifold $X$ is called \emph{transversally elliptic} if the support of the symbol $\sigma(P)$ (the subset of $T^\ast X$ where $\sigma(P)$ fails to be invertible) intersected with the conormal space to the orbits 
\[T_G^\ast X=\{\xi \in T^\ast X|\pair{\xi}{\alpha_X}=0, \forall \alpha \in \g \} \]
is compact.  References for transversally elliptic operators include \cite{AtiyahTransEll, ParVerTransEll}.

\begin{proposition}
\label{prop:TransEll}
For $t \in [0,1)$ the operator $P_t$ is transversally elliptic.
\end{proposition} 
\begin{proof}
The symbol of $P_t$ is a bundle map $\V^+ \rightarrow \V^-$ given by\footnote{Here we revert to the more common convention for symbols, such that $\sigma_{\st{D}^M}(\xi)=\i \c(\xi)$.}
\begin{equation} 
\label{eqn:SymbolPt}
\sigma_{P_t}(\xi)=(1-t)\psi^++\big(\i\rho\wh{\otimes}\langle\xi \rangle^{-1}\c(\xi)\big)\oplus 0_F, 
\end{equation}
Outside $M$ this equals $(1-t)\psi^+$, which is invertible for $t \ne 1$.  On $M$, $\rho=h^{-1}$ and $\psi=(2f\theta\wh{\otimes}1-\i\wh{\otimes} \c(v_M)) \oplus \id_F$, thus
\[ \sigma_{P_t}(\xi)|_M=\Big(2(1-t)f\theta\wh{\otimes}1+\i \rho \wh{\otimes} \c\big(\langle\xi \rangle^{-1}\xi-(1-t)h v_M\big)\Big) \oplus \id_F.\]
The expression in the large brackets is a product symbol (see for example \cite[p.56]{BerlineVergneTransversallyElliptic} for a discussion of products), and the support of a product symbol is the intersection of the supports of the two factors.  Since $v_M$ is tangent to the orbit directions, the intersection of the support of $$\c\big(\langle\xi \rangle^{-1}\xi-(1-t)h v_M\big)$$ with $T_T^\ast M$ (the conormal directions to the orbits) is the vanishing locus of $v_M$ (viewed as a subset of the zero section of $T^\ast_TM$), and the latter intersects the support of $(1-t)f\theta$ in a closed subset of the compact set $U$.
\end{proof}

We remind the reader of some properties of the operator $\scr{D}_1^M$ on the cylindrical end, to be used in the proof of the next lemma.  Recall from the proof of Proposition \ref{prop:DMFredholm} that, on the $\lambda$-isotypical subspace, one has an inequality
\begin{equation} 
\label{eqn:ineqonM}
(\scr{D}^M_1)^2_\lambda \ge ((\st{D}^M)^2+V)_\lambda 
\end{equation}
where $V$ is a potential function (depending on $\lambda$).  Moreover there exists an $r_0>1$, and constant $V_0>0$ such that 
\begin{equation}
\label{eqn:ineqVonW}
V|_{N\times [r_0,\infty)} \ge V_0e^{2r}.
\end{equation}
Let
\[ W=N\times [r_0,\infty),\]
a compact manifold with boundary, and let 
\[ \scr{D}_1^W=\scr{D}^M_1|_W \]
be the restriction of $\scr{D}^M_1$ to $W$.  Hence
\begin{equation}
\label{eqn:DonW}
\scr{D}^W_1=\st{D}^W+e^r\psi, \qquad \st{D}^W=\c(\partial_r)\partial_r+\st{D}^N, 
\end{equation}
where $\psi$ is the self-adjoint bundle endomorphism introduced in \eqref{eqn:psiDM}, $\st{D}^W$ is the restriction of $\st{D}^M$ to $W$, and $\st{D}^N=-\c(\partial_r)A^N$ is a Dirac operator on $N$ not depending on $r$. Recall $\ol{M}$ is a compact subset of $\Y$ and $\ol{M}\cap U_\gamma=\emptyset$ for $\gamma \ne \beta$, so that $\psi$ is invertible on the compact set $\ol{M}\setminus \tn{int}(U)\supset W$.  It follows from this and $\psi=\psi^\ast$ that there are constants $0<c_\psi\le C_\psi$ such that
\begin{equation}
\label{eqn:boundpsi}
c_\psi^2|v|^2\le \pair{\psi^2 v}{v}\le C_\psi^2 |v|^2 
\end{equation}
for all $v \in \V_m$ and $m \in W$.

The square of $\scr{D}^W_1$ is given by
\begin{equation} 
\label{eqn:squareonW}
(\scr{D}^W_1)^2=(\st{D}^W)^2+[\st{D}^W,e^r\psi]+e^{2r}\psi^2.
\end{equation}
The operator $[\st{D}^W,e^r\psi]$ commutes with $T$-invariant functions, and moreover \eqref{eqn:ineqonM}, \eqref{eqn:ineqVonW} imply
\begin{equation} 
\label{eqn:ineqonW}
([\st{D}^W,e^r\psi]+e^{2r}\psi^2)_\lambda \ge V_0e^{2r}.
\end{equation}
 
The next lemma is \cite[Lemma 2.5]{MaZhangBravermanIndex} due to Ma and Zhang.  It shows that solutions of $\scr{D}_1^Ms=0$ decay very rapidly, as well as their first derivatives, as $r\rightarrow \infty$ along the cylindrical end. For the convenience of the reader we have included a proof here, in the notation that we have established, closely following the argument in \emph{loc. cit}.  
\begin{lemma}[\cite{MaZhangBravermanIndex}, Lemma 2.5]
\label{lem:SqrInt}
Let $W=N\times (r_0,\infty) \subset M$ be as in the paragraph above.  Let $s$ be a smooth section of $\E$ over $W$ and suppose $e^{-r}s \in L^2(W,\E)_\lambda$ and $\scr{D}^W_1s=0$.  Then 
\[ e^{mr}s, \quad e^{mr}\partial_r s, \quad e^{mr}\st{D}^Ns \in L^2(W,\E)_\lambda\] 
for all $m \in \bR$.
\end{lemma}
\begin{remark}
Once the square-integrability of $e^{mr}s$ is in hand, the square-integrability of $e^{mr}\partial_rs, e^{mr}\st{D}^Ns$ is obtained by `bootstrapping' using $\st{D}^Ws=-e^r\psi s$, and indeed one could obtain the same result for higher derivatives of $s$ as well in this way.  Heuristically one might expect the behavior in Lemma \ref{lem:SqrInt} because if $c \ne 0$ is a constant then the ODE $a'(r)=ce^ra(r)$ has solutions of the form $a(r)=\exp(ce^r)$; if $c>0$ then $e^{mr}a^{(n)}(r)$ is not square-integrable on $[0,\infty)$ for any $m$, whereas if $c<0$ then $e^{mr}a^{(n)}(r)$ is square-integrable on $[0,\infty)$ for every $m$.
\end{remark}
\begin{proof}
Let $w=w(r) \in C^\infty_c(\bR)$ be a non-negative function, which we view as a function on $W$.  Using $\scr{D}^W_1s=0$ and since $[\st{D}^W,e^r\psi]$ commutes with $w$ and satisfies \eqref{eqn:ineqonW}, we have
\begin{align}
\label{eqn:step1} 
0&=((\scr{D}^W_1)^2s,w^2s)=\big((\st{D}^W)^2s,w^2s\big)+\big(([\st{D}^W,e^r\psi]+e^{2r}\psi^2)ws,ws\big)\nonumber \\
&\ge \tn{Re}\big((\st{D}^W)^2s,w^2s\big)+V_0(e^{2r}ws,ws).
\end{align}
We apply Green's formula to the first term, and also use $\st{D}^Ws=-e^r\psi s$ (twice), leading to
\begin{align*} 
\tn{Re}\big((\st{D}^W)^2s,w^2s\big)&=\tn{Re}(\st{D}^Ws,w^2\st{D}^Ws+[\st{D}^W,w^2]s)+B_w\\
&=\tn{Re}(e^r\psi s,w^2e^r\psi s-2ww'\c(\partial_r)s)+B_w\\
&\ge c_\psi^2\|e^r ws\|^2-2C_\psi\|e^{r/2}ws\|\cdot \|e^{r/2}w's\|+B_w
\end{align*}
where $B_w$ is a boundary term, involving an integral over $\partial W=N\times \{r_0\}$; and in the last line we used $\psi^\ast \psi\ge c_\psi^2$, the Cauchy-Schwartz inequality, $|\c(\partial_r)s|=|s|$ and $|\psi|\le C_\psi$.  Substituting this into \eqref{eqn:step1}, we have
\begin{equation}
\label{eqn:step2}
0\ge (c_\psi^2+V_0) \|e^r ws\|^2-2C_\psi\|e^{r/2}ws\|\cdot \|e^{r/2}w's\|+B_w.
\end{equation}

Let $\chi\colon \bR \rightarrow [0,1]$ be a bump function with support contained in $(r_0-1,r_0+1)$, and set $\chi_k(r)=\chi(k^{-1}(r-r_0)+r_0)$.  As $k \rightarrow \infty$, $\chi_k \rightarrow 1$ pointwise, and is centred about $r_0$.  We will view $\chi_k$ as functions on $W$.  Let $w(r)=e^{mr}\chi_k(r)$.  Then equation \eqref{eqn:step2} becomes
\begin{equation}
\label{eqn:step3}
0\ge (c_\psi^2+V_0) \|e^{(m+1)r}\chi_k s\|^2-2C_\psi\|e^{(m+1/2)r}\chi_k s\|\cdot \|e^{(m+1/2)r}(m\chi_k+\chi_k')s\|+B_w.
\end{equation}
The boundary term $B_w$ does not depend on $k$, since $\chi_k$ equals $1$ on a neighborhood of $r_0$ for all $k$.  If $|e^{(m+1/2)r}s|$ is square-integrable, then the middle term has a limit as $k \rightarrow \infty$, so if this is the case, then the inequality \eqref{eqn:step3} implies
\[ \lim_{k\rightarrow \infty} \|e^{(m+1)r}\chi_ks\| \]
exists, hence $|e^{(m+1)r}s|$ is square-integrable as well.  By induction, $|e^{mr}s|$ is square-integrable for all $m \in \bR$. 

Turning to the first derivatives, first note that since $\c(\partial_r)\partial_r$, $\st{D}^N$ graded commute, Green's formula gives an identity
\begin{equation} 
\label{eqn:greengrcomm}
\|\st{D}^W\alpha \|^2=\|\partial_r\alpha\|^2+\|\st{D}^N\alpha\|^2+B_\alpha',
\end{equation}
for $\alpha$ smooth and compactly supported, where $B_\alpha'$ is a boundary term.  Let $w=w(r) \in C^\infty_c(\bR)$, and consider the equation:
\[ -we^r\psi s=w\st{D}^W s.\]
Taking the $L^2$ norm-squared of both sides, we find
\begin{align*}
C_\psi^2 \|we^rs\|^2 &\ge \|w\st{D}^Ws\|^2\\
&\ge \tfrac{1}{2}\|\st{D}^W(ws)\|^2-\|[\st{D}^W,w]s\|^2\\
&=\tfrac{1}{2}\|\partial_r(ws)\|^2+\tfrac{1}{2}\|\st{D}^N(ws)\|^2+B_{ws}'-\|[\st{D}^W,w]s\|^2\\
&\ge \tfrac{1}{4}\|w\partial_r s\|^2+\tfrac{1}{2}\|w\st{D}^Ns\|^2-\tfrac{1}{2}\|w's\|^2+B_{ws}'-\|[\st{D}^W,w]s\|^2,
\end{align*}
where we used $\|a_1+a_2\|^2\ge \tfrac{1}{2}\|a_1\|^2-\|a_2\|^2$ in the second and fourth lines, and \eqref{eqn:greengrcomm} in the third line.  If we take $w=\chi_k e^{mr}$ as before, then the first part of the proof implies all terms in this expression other than those containing $\partial_rs$, $\st{D}^Ns$ are uniformly bounded in $k$ as $k \rightarrow \infty$ (and $B_{ws}'$ does not depend on $k$ as before).  Consequently
\[ \|e^{mr}\st{D}^Ns\|=\lim_{k\rightarrow \infty}\|\chi_k e^{mr}\st{D}^Ns\|, \qquad \|e^{mr}\partial_r s\|=\lim_{k\rightarrow \infty} \|\chi_k e^{mr}\partial_r s\| \]
are both finite.
\end{proof}
\ignore{
. By the first part of the proof, $\|e^{mr}s\|<\infty$, hence
\[ \|w's\|\le \tfrac{\|\chi'\|_\infty}{k}\|e^{mr}s\|+m\|\chi_k e^{mr}s\| \]
and we deduce that $\|w's\|^2$ is bounded above by some constant $C_m$ not depending on $k$.  Equation \eqref{eqn:step3} becomes
\begin{equation}
\label{eqn:step5}
\|\st{D}^W(ws)\|^2\ge \|w\st{D}^Ns\|^2+\tfrac{1}{2}\|w\partial_r s\|^2-C_m.
\end{equation}
On the other hand, using 
\[ \st{D}^W(ws)=w'\c(\partial_r)s+w\st{D}^Ws=w'\c(\partial_r)s-we^r\psi s\]
we obtain an upper bound
\[ \|\st{D}^W(ws)\|^2\le 2\|w's\|^2+2C_\psi^2\|we^rs\|^2.\]
Combining this with \eqref{eqn:step4} we find
\[ 2\|w's\|^2+2C_\psi^2\|we^rs\|^2\ge \|w\st{D}^Ns\|^2+\tfrac{1}{2}\|w\partial_r s\|^2-C_m. \]
Recall $w=\chi_k e^{mr}$. The left hand side has a finite limit as $k \rightarrow \infty$, by the first part of the proof.  
\end{proof}
}
\begin{proposition}[cf. \cite{Braverman2002} Section 14.6, \cite{MaZhangBravermanIndex} Section 3]
\label{prop:bravindex}
Sections in the kernel of $P_1$ or $P_1^\ast$ vanish identically outside $M$.  For each $\lambda \in \Irr(T)$, there are isomorphisms
\[ \ker(P_1)_\lambda \simeq \ker(\scr{D}_1^{M,+})_\lambda,\qquad \coker(P_1)_\lambda=\ker(P_1^\ast)_\lambda\simeq \ker(\scr{D}_1^{M,-})_\lambda.\]
Therefore $P_1$ is $T$-Fredholm and
\[ \index(P_1)=\index(\scr{D}_1^{M}) \in R^{-\infty}(T).\]
\end{proposition}
\begin{proof}
For $t=1$
\[ P_1=(\psi^+ +\rho \st{D}^{M,+})R\]
hence $R$ induces an isomorphism between $\ker(P_1)$ and $\ker(\psi^++\rho \st{D}^{M,+})$, where $\psi^++\rho \st{D}^{M,+}$ has domain $H^1(DM,\V^+)$.  
Let $s \in \ker(\psi^++\rho \st{D}^{M,+})_\lambda$, then $s$ must vanish outside $M$, since outside $M$, $\rho=0$ while $\psi^+$ is invertible.  On $M$ we have
\begin{equation} 
\label{eqn:psiDM}
\psi^++\rho \st{D}^{M,+}=\rho \scr{D}^{M,+}_1\oplus \id_F,
\end{equation}
hence the $F$-component of $s$ vanishes and
\[ \rho \scr{D}^{M,+}_1 s=0.\]
Since $\rho>0$ on $M$, this implies 
\[ \scr{D}^{M,+}_1 s=0, \]
that is, $s$ lies in the kernel of the differential operator $\scr{D}^{M,+}_1$.  By elliptic regularity $s$ is smooth.  Moreover $s$ is $g_{DM}$-square integrable, hence by \eqref{eqn:DMmetricw}, $r^{-1}s|_{\Cyl_N}$ is $g_M|_{\Cyl_N}$-square integrable.  Applying Lemma \ref{lem:SqrInt}, $s$ is $g_M$-square integrable, hence lies in $\dom(\scr{D}^{M,+}_1)$.  This shows that there is an inclusion $\ker(\psi^++\rho\st{D}^{M,+})_\lambda \hookrightarrow \ker(\scr{D}^{M,+}_1)_\lambda$. 

Conversely suppose $s \in \ker(\scr{D}^{M,+}_1)_\lambda$.  Let $\tilde{s}$ be the extension of $s$ by $0$ to $DM$.  Lemma \ref{lem:SqrInt} shows that $\tilde{s}$ lies in the Sobolev space $H^1(DM,\V^+)=\dom(\psi^++\rho\st{D}^{M,+})$, and hence in $\ker(\psi^++\rho\st{D}^{M,+})$ by \eqref{eqn:psiDM}.  Combining this with the previous paragraph, we have identified $\ker(\psi^++\rho\st{D}^{M,+})_\lambda$ with $\ker(\scr{D}^{M,+}_1)_\lambda$.

The adjoint is
\[ P_1^\ast=R(\psi^-+\st{D}^{M,-}\circ \rho),\]
where the operator $\st{D}^{M,-}\circ \rho$ is the adjoint of $\rho\st{D}^{M,+}$. Let $s \in \ker(P_1^\ast)_\lambda$.  Since $R$ is invertible, we must have
\begin{equation} 
\label{eqn:dist}
(\psi^-+\st{D}^{M,-}\circ \rho)s=0
\end{equation}
in the sense of distributions. Similar to above $s$ must vanish outside $M$, its $F$-component vanishes, and
\[ \scr{D}^{M,-}_1(\rho s)=0.\]
By elliptic regularity, $\rho s|_M$ is smooth.  Since $s$ is $g_{DM}$-square-integrable, $\rho s|_M$ is $g$-square integrable.  Thus $\rho s|_M \in \ker(\scr{D}^{M,-}_1)_\lambda$.

Conversely suppose $\tilde{s} \in \ker(\scr{D}^{M,-}_1)_\lambda$.  Extending $h\tilde{s}$ by $0$, we obtain a section $s$ over $DM$.  By Lemma \ref{lem:SqrInt}, $s \in H^1(DM,\V^-)$.  Since $\rho|_M=h^{-1}$,
\[ \scr{D}^{M,-}_1(\rho h \tilde{s})=\scr{D}^{M,-}_1\tilde{s}=0 \]
and this implies \eqref{eqn:dist} holds in the sense of distributions, hence $s \in \ker(P_1^\ast)_\lambda$.  This proves the map $s \mapsto \rho s|_{M}$ identifies $\ker(P_1^\ast)_\lambda$ with $\ker(\scr{D}^{M,-}_1)_\lambda$.
\end{proof}
\begin{corollary}
The $T$-index $\index(P_t) \in R^{-\infty}(T)$ is independent of $t \in [0,1]$.
\end{corollary}
\begin{proof}
For each $\lambda \in \Lambda^\ast$, the restriction of $P_t$ to the $\lambda$-isotypical component is a norm-continuous family of Fredholm operators, hence the index is constant.
\end{proof}

\subsection{Abelian localization and the Paradan-type formula.}
For a compact $G$-manifold $X$, the symbol $\sigma(P)$ of a transversally elliptic operator $P \colon C^\infty(X,E) \rightarrow C^\infty(X,F)$ defines a class in $\K^0_G(T_G^\ast X)$; this class is given in terms of the `difference bundle construction' cf. \cite{AtiyahKtheory}, and only depends on the bundles $E$, $F$ and the behavior of $\sigma(P)$ away from the $0$-section.  If $X$ is compact, then the index map
\begin{equation} 
\label{eqn:IndexMapSymbols}
\index \colon \K^0_G(T_G^\ast X)\rightarrow R^{-\infty}(G) 
\end{equation}
is defined by realizing elements of $\K^0_G(T_G^\ast X)$ as symbols of transversally elliptic operators on $X$ via the difference bundle construction, followed by taking the analytic index, see \cite{AtiyahTransEll, ParVerTransEll}.  If $X$ is non-compact, then the index map \eqref{eqn:IndexMapSymbols} is defined by first embedding $X$ into a compact $G$-manifold $X^\prime$, and choosing suitable representatives of K-theory classes which can be extended by the identity outside $X$.  We did exactly this for $X=M$ and the bundle morphism $\psi^\prime$ (and hence also the symbol $\sigma_{P_0}|_M$ in equation \eqref{eqn:SymbolPt}) near the beginning of Section \ref{sec:DefTransEll}, with the compact manifold $X^\prime$ being the double $DM$.

Let $\mathring{U}=\mathring{U}_\beta$ denote the interior of $U=U_\beta$.  The restriction of the symbol of $P_0$ to $\mathring{U}$ is a symbol $\sigma_0$ given by
\[ \sigma_0(\xi)=\big(2f\theta\wh{\otimes}1+\i\wh{\otimes} \c(\langle\xi \rangle^{-1}\xi-v_{\sY})\big)\oplus \id_F,\]
where the right hand side is viewed as a bundle map $\E^+\oplus F \rightarrow \E^-\oplus F$.  This symbol defines a class in $\K_T^0(T_T^\ast \mathring{U})$, and by the above discussion, its index is $\index(P_0)$.  Since the index depends only on the class in $\K_T^0(T_T^\ast \mathring{U})$, we may drop the $\id_F$ component (this represents the trivial element in K-theory), and we may use a (straight-line) homotopy to eliminate the factors $2f$, $\langle\xi \rangle$.  This leads to the following.

\begin{proposition}
\label{prop:TransEllSymb}
The symbol
\[ \ti{\sigma}_{\beta,\theta}(\xi)=\theta\wh{\otimes}1+\i \wh{\otimes}\c (\xi-v_{\sY})\]
on $\mathring{U}=\mathring{U}_\beta$ is $T$-transversally elliptic and
\[ \index(\ti{\sigma}_{\beta,\theta})=\index(P_0)=\index_{\tn{APS},\beta}(\scr{D},v). \]
\end{proposition}  
\ignore{
Moreover $\ti{\sigma}_{\beta,\theta}$ is a product of the symbol 
\[ \ti{\sigma}_\beta=\i\c\big(\xi-v_{\sY}\big)\] 
acting on $p^\ast S$ (where $p \colon T^\ast U \rightarrow U$), with the bundle endomorphism $\theta$.
This becomes clearer if we restore the graded tensor product notation that we have been mostly suppressing:
\begin{equation} 
\label{eqn:ProdSymb}
\ti{\sigma}_{\beta,\theta}(\xi)=\theta \wh{\otimes}1+\i\wh{\otimes}\c(\xi-v_{\sY}).
\end{equation}

if we present the morphism $\ti{\sigma}_{\beta,\theta}^+\colon \pi^\ast\E^+ \rightarrow \pi^\ast\E^-$ in terms of the direct sum decompositions
\[ \E^+=(E^+\otimes S^+)\oplus(E^-\otimes S^-), \qquad \E^-=(E^-\otimes S^+)\oplus(E^+\otimes S^-),\]
as the matrix
\[ \ti{\sigma}_{\beta,\theta}^+=\left(\begin{array}{cc} \theta^+ \otimes 1 & -1\otimes \ti{\sigma}_\beta^- \\ 1\otimes \ti{\sigma}_\beta^+ & \theta^-\otimes 1 \end{array}\right).\]

To reflect this we introduce the notation $\theta \otimes \ti{\sigma}_\beta$ for the symbol $\ti{\sigma}_{\beta,\theta}$:
\begin{equation}
\label{def:ProdSymb}
\ti{\sigma}_{\beta,\theta}=\theta \otimes \ti{\sigma}_\beta.
\end{equation}
}
For $\beta \ne 0$, let $\nu_\beta=\nu(\mathring{U},\mathring{U}^\beta)$ be the normal bundle to the fixed-point set $\mathring{U}^\beta$.  It inherits a metric by identifying $\nu_\beta$ with the Riemannian orthogonal complement of $T\mathring{U}^\beta$.  Let $T_\beta \subset T$ denote the subtorus obtained by taking the closure of $\exp_T(\bR\beta)$.  Then $T_\beta$ fixes $\mathring{U}^\beta$, so acts fibre-wise on the normal bundle $\nu_\beta$.  We may choose a complex structure on $\nu_\beta$ such that the complex $T_\beta$-weights are $\beta$-\emph{polarized}, i.e. for each complex weight $\alpha$ of the $T_\beta$ action on $\nu_\beta$ one has $\pair{\alpha}{\beta}>0$.  This condition determines the complex structure on $\nu_\beta$ up to homotopy.  Let $\wedge \nu_\beta$ (resp. $\Sym(\nu_\beta)$) denote the complex exterior algebra (resp. complex symmetric algebra) bundle.  Let $\ol{\nu}_\beta$ denote $\nu_\beta$ equipped with the opposite complex structure; one has likewise $\wedge \ol{\nu}_\beta$ and $\Sym(\ol{\nu}_\beta)$.

The exterior algebra $\wedge \ol{\nu}_\beta$ is a spinor module for the Euclidean vector bundle $\nu_\beta \rightarrow \mathring{U}^{\beta}$.  One has a short exact sequence
\[ 0 \rightarrow T\mathring{U}^{\beta} \rightarrow T\mathring{U} \restriction \mathring{U}^{\beta} \rightarrow \nu_\beta \rightarrow 0.\]
By the 2-out-of-3 property for spin-c structures, the spinor modules $S$ for $T\mathring{U}$ and $\wedge \ol{\nu}_\beta$ for $\nu_\beta$ determine a $\bZ_2$-graded spinor module $S_\beta$ for $T\mathring{U}^{\beta}$ such that
\begin{equation}
\label{eqn:InducedSpinc}
S_\beta \otimes \wedge \ol{\nu}_\beta\simeq S \restriction \mathring{U}^{\beta}.
\end{equation}
For the corresponding determinant line bundles, equation \eqref{eqn:InducedSpinc} implies
\begin{equation}
\label{eqn:InducedDetLine}
\det(S_\beta)=\det(S)\otimes \det(\nu_\beta).
\end{equation}

There is a symbol $\sigma_{\beta,\theta}$ on $\mathring{U}_\beta^{T_\beta}$ defined in a similar manner to $\ti{\sigma}_{\beta,\theta}$:
\begin{equation} 
\label{eqn:ProdSymb}
\sigma_{\beta,\theta}(\xi)=\theta \wh{\otimes}1+\i\wh{\otimes}\c_\beta(\xi-v_{\sY}),
\end{equation}
where $\c_\beta$ denotes Clifford multiplication for the spinor module $S_\beta$, and the right hand side of \eqref{eqn:ProdSymb} is viewed as a bundle map $(E\wh{\otimes}S_\beta)^+ \rightarrow (E\wh{\otimes}S_\beta)^-$.  The symbol $\sigma_{\beta,\theta}$ defines a class in $\K^0_T(T^\ast_T\mathring{U}^\beta)$, so has an index.

The next proposition follows from an abelian localization theorem for transversally elliptic symbols due to Paradan \cite[Theorem 5.8, Proposition 6.4]{ParadanRiemannRoch} (see also \cite{WittenNonAbelian}), building on results of Atiyah \cite{AtiyahTransEll} and Berline-Vergne.
\begin{proposition} $\index(\ti{\sigma}_{\beta,\theta})=\index(\sigma_{\beta,\theta} \otimes \Sym(\nu_\beta))$.
\end{proposition}
Here $\index(\sigma \otimes \Sym(\nu))$ is defined as the sum over $k\ge 0$ of $\index(\sigma \otimes \Sym^k(\nu))$, the index of the transversally elliptic symbol $\sigma$ twisted by the finite dimensional vector bundle $\Sym^k(\nu)$ (the $k^{th}$ symmetric power).  As a corollary, we obtain a Paradan-type (\cite{ParadanRiemannRoch,WittenNonAbelian}) `norm-square localization' formula for $\index(\scr{D})$.
\begin{theorem}
\label{thm:NormSqrLoc}
We have the following equality in $R^{-\infty}(T)$
\begin{equation} 
\label{eqn:NormSqrLoc}
\index(\scr{D})=\sum_\beta \index_{\tn{APS},\beta}(\scr{D},v)=\sum_{\beta} \index(\sigma_{\beta,\theta} \otimes \Sym(\nu_\beta)).
\end{equation}
The sum is over $\beta \in \t$ labelling components $Z_\beta$ of the vanishing locus $Z$; in other words, the sum is over $\beta \in \t$ such that $Z_\beta=\Y^\beta \cap \phi^{-1}(\beta) \ne \emptyset$.  This is an infinite discrete subset of $\t$.
\end{theorem}

\subsection{Remarks on the $[Q,R]=0$ theorem for loop group spaces.}\label{ssec:QRrem}
In this section we briefly comment on the relation between Theorem \ref{thm:NormSqrLoc} and the $[Q,R]=0$ theorem for Hamiltonian loop group spaces.  The relationship between Paradan-type formulas (as in \eqref{eqn:NormSqrLoc}) and $[Q,R]=0$ theorems goes back to the work of Paradan \cite{ParadanRiemannRoch} (see also Paradan and Vergne \cite{WittenNonAbelian}).

Throughout this section we assume $G$ is simple and simply connected, and that the inner product on $\g$ is the \emph{basic inner product}, the unique invariant inner product normalized such that the squared lengths of the short co-roots is $2$.  The possible $U(1)$ central extensions of $LG$ are classified by an integer known as the \emph{level}.  Let $\wh{LG}$ denote the level $1$ central extension, sometimes called the \emph{basic} central extension.  By restriction we obtain a central extension $N_G(T)\ltimes \wh{\Lambda}$.  It satisfies \eqref{eqn:CommutationRelation} with the homomorphism $\kappa$ being the musical isomorphism induced by the basic inner product.

Let $\Phi_{\M} \colon \M \rightarrow L\g^\ast$ be a proper Hamiltonian $LG$-space.  A vector bundle $E \rightarrow \M$ is said to be at level $k \in \bZ$ if $E$ is $\wh{LG}$-equivariant and the central circle acts with weight $k$.  A level $k>0$ prequantum line bundle $L \rightarrow \M$ is a line bundle at level $k$ with invariant connection $\nabla^L$, where the first Chern form $c_1(\nabla^L)=k\omega_\M$, and $\nabla^L$ satisfies Kostant's condition (cf. \cite{AMWVerlinde}):
\[ \L^L_{\xi}-\nabla^L_{\xi_\M}=2\pi \i k\pair{\Phi_\M}{\xi}, \qquad \xi \in L\g\oplus 0 \subset \wh{L\g}.\]

In joint work with E. Meinrenken \cite{LMSspinor} we constructed a canonical spinor module $S_0$ for $\Cliff(p^\ast TM)$, where $p \colon \M \rightarrow M=\M/\Omega G$ is the quotient map; in \cite{LMSspinor} $S_0$ was referred to as a `twisted spin-c structure' for $M$.  $S_0$ is at level $\hvee$, the dual Coxeter number of $G$.  Let $S=S_0 \otimes L$, a spinor module for $p^\ast TM$ at level $k+\hvee$.

Recall that the `global transversal' $\Y$ embeds $N_G(T)\ltimes \Lambda$-equivariantly into $\M$ (Section \ref{sec:GlobTrans}), with $Tp$ inducing an isomorphism $T\Y \simeq p^\ast TM|_{\Y}$.  Hence by restriction to $\Y$ we obtain an $N_G(T)\ltimes \wh{\Lambda}$-equivariant spinor module for $\Y$, that we also denote by $S$.  Choosing a compatible $N_G(T)\ltimes \wh{\Lambda}$-invariant connection, we obtain a spin-c Dirac operator $\st{D}^S$ acting on sections of $S$.

Let $\n_- \subset \g_{\bC}$ denote the sum of the negative root spaces of $\g$, and let $\Bott(\n_-)=[(\wedge \n_-,\theta)] \in \K^0_T(\g/\t)$ denote the Bott-Thom element for $\g/\t$; here $\theta$ is an odd self-adjoint endomorphism of $(\g/\t)\times \wedge \n_-$, invertible away from the origin (cf. \cite[Section 4.5]{LSQuantLG}).  We may choose $\theta$ so that the pullback via $\phi^{\g/\t}$ of the pair $(\wedge \n_-,\theta)$ satisfies the conditions of Section \ref{sec:FirstOrdY}.  Let $\st{D}$ denote the Dirac operator acting on sections of $\wedge \n_- \wh{\otimes}S$ obtained by coupling $\st{D}^S$ to the $\bZ_2$-graded bundle $E=\Y \times \wedge \n_-$, and $\scr{D}$ the operator described in Theorem \ref{thm:FredholmI}.  According to the latter theorem, $\scr{D}$ is $T$-Fredholm.

With this setup, we may state a version of the $[Q,R]=0$ Theorem for proper Hamiltonian $LG$-spaces.  For simplicity suppose $G$ acts freely on $\Phi_{\M}^{-1}(0)$, so that the reduced space $\M_{\tn{red}}=\Phi_{\M}^{-1}(0)/G$ is a smooth, finite-dimensional compact symplectic manifold with prequantum line bundle $L_{\tn{red}}=L|_{\Phi^{-1}_{\M}(0)}/G$.  Choose a compatible almost complex structure on $\M_{\tn{red}}$ and let $\dirac$ denote the Dolbeault-Dirac operator twisted by $L_{\tn{red}}$ acting on $\wedge T^\ast_{0,1}\M_{\tn{red}}\otimes L_{\tn{red}}$.
\begin{theorem}
\label{thm:QR0}
Let $\Phi_{\M} \colon \M \rightarrow L\g^\ast$ be a proper Hamiltonian $LG$-space, with level $k>0$ prequantum line bundle $L$.  Assume $G$ acts freely on $\Phi_{\M}^{-1}(0)$.  Let $\scr{D}$, $\dirac$ be the operators described above.  Then $\index(\scr{D})_0=\index(\dirac)$.
\end{theorem}
\begin{remark}
There is a similar statement when $0$ is a regular value of $\Phi_{\M}$, in which case $\M_{\tn{red}}$ is only an orbifold in general.  There is also a statement when $0$ is not necessarily a regular value of $\Phi_{\M}$, involving a shift (partial) desingularization as in \cite{MeinrenkenSjamaar}.
\end{remark}

Let us reformulate Theorem \ref{thm:QR0} to highlight its similarity with other instances of the $[Q,R]=0$ phenomenon.  We define the `quantization' of the finite-dimensional symplectic manifold $\M_{\tn{red}}$ to be the `Riemann-Roch number':
\begin{equation} 
\label{eqn:DefineQred}
Q(\M_{\tn{red}},L_{\tn{red}})=\index(\dirac) \in \bZ.
\end{equation}

Let $R_k(G)$ denote the level $k$ \emph{fusion ring} (or \emph{Verlinde algebra}), a finite rank $\bZ$-module (and ring) generated by the irreducible level $k$ positive energy representations of $\wh{LG}$ (cf. \cite{PressleySegal}, \cite[Appendix D]{MeinrenkenKHomology}).  A positive energy representation is, in particular, a representation of the semi-direct product $S^1_{\tn{rot}}\ltimes \wh{LG}$ ($S^1_{\tn{rot}}$ acts on $LG$ by loop rotation, and this action lifts to an action on $\wh{LG}$).  Elements $V \in R_k(G)$ have formal characters $\tn{ch}\, V \in R^{-\infty}(S^1_{\tn{rot}}\times T)$, given by the Weyl-Kac character formula \cite{PressleySegal, KacBook}.

In \cite[Section 4.5]{LSQuantLG} we proved that $\index(\scr{D}) \in R^{-\infty}(T)$ is anti-symmetric for the $\rho$-shifted level $(k+\hvee)$ action of the affine Weyl group on $\Lambda^\ast$, given by
\[ w\bullet_{k+\hvee} \lambda = \ol{w}(\lambda+\rho)-\rho + (k+\hvee) \kappa(\eta), \qquad w=(\ol{w},\eta) \in W \ltimes \Lambda=W_{\tn{aff}},\]
where recall $\kappa \colon \Lambda \rightarrow \Lambda^\ast$ is the map induced by the musical isomorphism $\t \rightarrow \t^\ast$ for the basic inner product.

Recall that by the Weyl character formula, characters of $G$ are in 1-1 correspondence with $\rho$-shifted $W$-anti-symmetric characters of $T$.  Likewise by the Weyl-Kac character formula, the above $W_{\aff}$-anti-symmetry implies that there is a unique element of the level $k$ fusion ring $Q(\M,L) \in R_k(G)$ such that
\begin{equation} 
\label{eqn:DefineQ}
\Big(\Delta \cdot \tn{ch}\,Q(\M,L)\Big)\Big|_{q=1}=\index(\scr{D}) \in R^{-\infty}(T)
\end{equation}
where $\Delta=\prod_{\alpha \in \R_{\tn{aff},+}}(1-e_{-\alpha})$ is the Weyl-Kac denominator, and we restrict to $q=1 \in S^1_{\tn{rot}}$.  Partly motivated by this, in \cite[Section 4.6]{LSQuantLG} (combined with the result in \cite[Section 4.7]{LSQuantLG}), we took \eqref{eqn:DefineQ} as the \emph{definition} of the quantization $Q(\M,L)$ of $(\M,L)$; see there for details.

The \emph{minimal} irreducible positive energy representation of $\wh{LG}$ at level $k$ is the one labelled by the highest weight $(k,0) \in \bN \times \Lambda^\ast$, cf. \cite{PressleySegal}.  It follows from the Weyl-Kac formula that the multiplicity of the minimal irreducible representation in $V \in R_k(G)$ is equal to the multiplicity of the trivial representation in $\Delta \cdot \tn{ch}\, V|_{q=1}$.  Thus Theorem \ref{thm:QR0} is equivalent to: 
\begin{corollary}
Let $(\M,L)$ be as in Theorem \ref{thm:QR0}.  Then $Q(\M_{\tn{red}},L_{\tn{red}})$ equals the multiplicity of the minimal level $k$ irreducible positive energy representation in $Q(\M,L)$.
\end{corollary}
\begin{remark}
A result of \cite{LoizidesGeomKHom} proves that the definition of the quantization of $(\M,L)$ given here (or \cite{LSQuantLG}) is equivalent to the definition of E. Meinrenken \cite{MeinrenkenKHomology} (the latter approach is via q-Hamiltonian spaces, twisted K-homology, and the Freed-Hopkins-Teleman theorem).  Thus Theorem \ref{thm:QR0} implies the $[Q,R]=0$ Theorem for that definition as well.  The latter theorem had been proven much earlier in \cite{AMWVerlinde} using symplectic cutting techniques and detailed analysis of the fixed-point expressions (in fact in \cite{AMWVerlinde}, Atiyah-Segal-Singer fixed-point expressions were used as a make-shift \emph{definition} of the quantization of $(\M,L)$, see \cite{MeinrenkenKHomology} for further explanation).
\end{remark}

The complete proof of Theorem \ref{thm:QR0} is not presented here, as it would take us too far from the main topics of this article.  The missing arguments are either already explained in the literature, or will be explained in \cite{LMVerlindeQR} (see also \cite{YiannisThesis}).  We settle for brief remarks on the main steps (1)--(4):

\medskip
\noindent (1) \textbf{A local $[Q,R]=0$ result.}\\
The transversally elliptic symbol $\sigma_{0,\theta}$ is defined on a small open neighborhood $\mathring{U}_0$ of $\Phi_{\M}^{-1}(0)$ in $\Y$ (viewing $\Y$ as a submanifold of $\M$).  For this reason, it is not too difficult to relate its index to the index of $\dirac$ on the reduced space using, for example, a local normal form near $\Phi_{\M}^{-1}(0)$.  By the cross-section theorem for Hamiltonian loop group spaces \cite[Theorem 4.8]{MWVerlindeFactorization}, one has the same local normal forms for $\mathring{U}_0$ available as in the case of a compact Hamiltonian $G$-space, hence this part of the argument is similar to the argument for compact Hamiltonian $G$-spaces explained in \cite[Theorem 8.3, Proposition 12.5]{WittenNonAbelian} or \cite[Section 6]{ParadanRiemannRoch},\footnote{These papers also handle the singular case, which carries over to our setting as well.} and leads to:
\begin{proposition} 
\label{prop:localQR}
$\index(\dirac)=\index(\sigma_{0,\theta})_0$.
\end{proposition}
\noindent In other words, $\index(\dirac)$ equals the multiplicity of the trivial representation in $\index(\sigma_{0,\theta})$, the contribution of $\beta=0$ in Theorem \ref{thm:NormSqrLoc}.

\medskip
\noindent (2) \textbf{A vanishing result when $\X^\beta \cap \phi^{-1}(\beta)=\emptyset$.}
\begin{proposition}
\label{prop:SuppBott}
$\index(\sigma_{\beta,\theta}\otimes \Sym(\nu_\beta))$ vanishes unless $\X^\beta \cap \phi^{-1}(\beta)\ne \emptyset$.  
\end{proposition}
\begin{proof}
The Bott element $\Bott(\n_-)$ is supported at $0 \in \g/\t$, so its pullback to $\Y$ is supported on $\X=(\phi^{\g/\t})^{-1}(0)$.  Recall $\sigma_{\beta,\theta}$ is a product symbol (see \eqref{eqn:ProdSymb}) of the pullback $\theta$ of the Bott symbol (supported on $\X$) with the symbol $\i \c_\beta(\xi-v_{\sY})$ (supported on $\Y^\beta \cap \phi^{-1}(\beta)$).  The support of a product symbol is the intersection of the supports, so $\sigma_{\beta,\theta}$ is supported on $\X^\beta \cap \phi^{-1}(\beta)$.  If the latter is empty then the support of $\sigma_{\beta,\theta}$ is empty, and the index vanishes.
\end{proof}

As a small aside, recall that $\X \simeq \Phi_{\M}^{-1}(\t)\subset \M$, and moreover $\Phi_{\M}|_{\X}=\phi|_{\X}$.  As $\beta \in \t$ we have $\X^\beta \cap \phi^{-1}(\beta)=\M^\beta \cap \Phi_{\M}^{-1}(\beta)$.  It follows from this that ($W$-orbits) of non-trivial contributions in \eqref{eqn:NormSqrLoc} correspond to the components of the critical set of $\|\Phi_\M\|^2$, cf. \cite{Kirwan,BottTolmanWeitsman,DHNormSquare}.

\medskip
\noindent (3) \textbf{Bounds on the support of $\index(\sigma_{\beta,\theta}\otimes \Sym(\nu_\beta))$ for $\beta \ne 0$.}\\
Let $\beta \ne 0$.  Since $T_\beta$ acts trivially on $(U_\beta)^\beta$, and because the weights of the $T_\beta$ action on $\Sym(\nu_\beta)$ are $\beta$-polarized, it follows (cf. \cite{ParadanRiemannRoch, WittenNonAbelian} for similar discussions) that the multiplicity function for $\index(\sigma_{\beta,\theta} \otimes \Sym(\nu_\beta))$ is supported in a half-space of the form $\{\xi \in \t|\pair{\xi}{\beta} \ge d_\beta \}$, where $d_\beta$ is a constant given by
\begin{equation} 
\label{eqn:dbeta}
d_\beta = \inf_{\alpha \in \wt(\wedge \n_- \otimes S_\beta)} \pair{\alpha}{\beta},
\end{equation}
the infimum being taken over the set of complex weights for the action of $T_\beta$ on $\wedge \n_- \otimes S_\beta$.  One proves the following:
\begin{theorem}
\label{thm:Bounddbeta}
For each $\beta \in W \cdot \B$ such that $\X^\beta \cap \phi^{-1}(\beta) \ne \emptyset$, the constant $d_\beta>0$.
\end{theorem}
This will be proved in detail in \cite{LMVerlindeQR} (see also  \cite{YiannisThesis}); it relies on a more detailed local description of the spinor module $S_0$ (to get an expression for the constant $d_\beta$), and a slightly subtle inequality involving the data of the affine Lie algebra $\wh{L\g}$.  This inequality is perhaps the most interesting aspect of Theorem \ref{thm:Bounddbeta}; it plays the same role as the `magical inequality' in \cite{ParadanVergneSpinc}.

\medskip
\noindent (4) \textbf{Conclusion.}\\
By Proposition \ref{prop:SuppBott} and Theorem \ref{thm:Bounddbeta}, $\index(\sigma_{\beta,\theta}\otimes \Sym(\nu_\beta))_0=0$ unless $\beta=0$.  Theorem \ref{thm:NormSqrLoc} then gives $\index(\scr{D})_0=\index(\sigma_{0,\theta})_0$.  Theorem \ref{thm:QR0} now follows from Proposition \ref{prop:localQR}.

\bibliographystyle{amsplain} 
\providecommand{\bysame}{\leavevmode\hbox to3em{\hrulefill}\thinspace}
\providecommand{\MR}{\relax\ifhmode\unskip\space\fi MR }
\providecommand{\MRhref}[2]{%
  \href{http://www.ams.org/mathscinet-getitem?mr=#1}{#2}
}
\providecommand{\href}[2]{#2}

\end{document}